\documentclass[a4paper,10pt]{amsart}

\usepackage{amsthm,amsmath,amssymb}
\usepackage[T1]{fontenc}
\usepackage[english]{babel}
\usepackage{mathscinet}
\usepackage{graphics}
\usepackage{graphicx}
\usepackage[all,cmtip]{xy}
\usepackage{multicol}
\usepackage{multirow}
\usepackage{algorithm}
\usepackage{algorithmic}
\usepackage{type1ec}
\usepackage{type1cm}
\usepackage{fix-cm}
\usepackage[lowtilde]{url}
\usepackage{hyperref}
\usepackage{enumerate}
\usepackage{placeins}
\usepackage{subfigure}

\theoremstyle{plain} 
\newtheorem{theo}{Theorem}[section]  
\newtheorem{prop}[theo]{Proposition}
\newtheorem{cor}[theo]{Corollary}
\newtheorem{lemma}[theo]{Lemma}

\theoremstyle{definition}
\newtheorem{defi}[theo]{Definition}
\newtheorem{ex}[theo]{Example}
\newtheorem{caso}{Case}       

\theoremstyle{remark}
\newtheorem{remark}[theo]{Remark}
\newtheorem{nothing}[theo]{\noindent\!\!\bf}
\newtheorem{notation}[theo]{Notation}
       
\DeclareMathOperator{\Sing}{Sing}

\DeclareMathOperator{\ord}{ord}

\DeclareMathOperator{\Mat}{Mat}

\DeclareMathOperator{\WeDiv}{WeDiv}
\DeclareMathOperator{\CaDiv}{CaDiv}
\DeclareMathOperator{\Div}{Div}
\DeclareMathOperator{\Pic}{Pic}
\DeclareMathOperator{\Cl}{Cl}

\DeclareMathOperator{\GL}{GL}
\DeclareMathOperator{\Supp}{Supp}
\DeclareMathOperator{\length}{length}
\DeclareMathOperator{\pr}{pr}

\newcommand\Z{\mathbb{Z}}
\newcommand\N{\mathbb{N}}
\newcommand\C{\mathbb{C}}
\newcommand\Q{\textbf{Q}}

\newcommand\bP{\mathbb{P}}
\newcommand\w{\omega}
\newcommand\E{\mathcal{E}}
\newcommand\cO{\mathcal{O}}
\newcommand\bd{\mathbf{d}}
\newcommand\ba{\mathbf{a}}
\newcommand\bxi{\boldsymbol{\xi}}

\title[Intersection Theory and $\Q$-Resolutions]{Intersection Theory on Abelian-Quotient $V$-Surfaces and $\Q$-Resolutions}

\author[E.~Artal]{Enrique Artal Bartolo}
\address[E.~Artal]{Departamento de Matem\'{a}ticas-IUMA  \\
Universidad de Zaragoza \\
C/~Pedro Cerbuna 12, 50009, Zaragoza, Spain}
\email{artal@unizar.es}

\author[J.~Mart\'{i}n-Morales]{Jorge Mart\'{i}n-Morales}
\address[J.~Mart\'{i}n-Morales]{Centro Universitario de la Defensa \\
Academia General Militar \\
Ctra.~de Huesca s/n. 50090, Zaragoza, Spain}
\email{jorge@unizar.es}
\urladdr{\url{http://cud.unizar.es/martin}}

\author[J.~Ortigas-Galindo]{Jorge Ortigas-Galindo}
\address[J.~Ortigas-Galindo]{Departamento de Matem\'{a}ticas-IUMA  \\
Universidad de Zaragoza \\
C/~Pedro Cerbuna 12, 50009, Zaragoza, Spain}
\email{jortigas@unizar.es}
\urladdr{\url{http://riemann.unizar.es/~jortigas/}}

\date{\today}

\begin{document}

\begin{abstract}
In this paper we study the intersection theory on surfaces with abelian quotient
singularities and we derive properties of quotients of weighted projective planes.
We also use this theory to study weighted blow-ups in order to construct
embedded $\Q$-resolutions of plane curve singularities 
and abstract $\Q$-resolutions of surfaces.
\end{abstract}

\maketitle
\setcounter{tocdepth}{1}
\tableofcontents

\section*{Introduction}
Intersection theory is a powerful tool in complex algebraic (and analytic) geometry, see~\cite{Fulton98} for 
a wonderful exposition. The case of smooth surfaces is of particular interest since
the intersection of objects is measured by integers.

The main objects involved in intersection theory on surfaces are divisors, which have two main incarnations,
Weil and Cartier. These coincide in the smooth case, but not in general. In the
singular case the two concepts are different and a geometric interpretation of intersection theory is yet
to be developed. A general definition for normal surfaces was given by Mumford~\cite{Mumford61}.

In this work we are interested in the intersection theory on $V$-surfaces with abelian quotient singularities.
We make use of our result in~\cite{kjj-qcw}, where we proved that the concepts of rational Weil and Cartier
divisors coincide. We will study their geometric properties and prove that the definition in this paper 
coincides with
Mumford's one. The most interesting points are the applications.

Probably the most well-known $V$-surfaces are the weighted projective planes. We will provide an extensive
study of intersection theory on these planes and on their quotients.

Closely related with the weighted
projective planes we have the weighted blow-ups. As opposed to standard ones, these blow-ups do not
produce smooth varieties, but the result may only have abelian quotient singularities. They can be used
to understand the birational properties of quotients of weighted projective planes and also
to obtain the so-called $\Q$-resolutions of singularities, where the usual
conditions are weakened: we allow the total space
to have abelian quotient singularities and the condition of normal crossing divisors
is replaced by $\mathbb{Q}$-normal crossing divisors. One of the main interest
of $\Q$-resolutions of singularities is the following: their combinatorial complexity
is extremely lower than the complexity of smooth resolutions, but they provide essentially
the same information for the properties of the singularity. 

Note that Veys has already studied this kind of embedded resolutions for plane curve
singularities, see \cite{Veys97}, in order to simplify the computation of the topological zeta function.

In both applications, we need intersection theory. Note that rational intersection numbers appear in
a natural way.

{\bf Acknowledgements.}
We thank J.I. Cogolludo for his fruitful conversations and ideas. All authors are partially supported by
the Spanish projects MTM2010-2010-21740-C02-02 and ``E15 Grupo Consolidado
Geometr{\'i}a'' from the government of Arag{\'o}n. Also, the second author,
J.~Mart{\'i}n-Morales, is partially supported by FQM-333, Junta de Andaluc{\'i}a.

\section{$V$-Manifolds and Quotient Singularities}

We sketch some definitions and properties, see~\cite{kjj-qcw} for a more detailed exposition.

\begin{defi}
A $V$-manifold of dimension $n$ is a complex analytic space which admits an open
covering $\{U_i\}$ such that $U_i$ is analytically isomorphic to $B_i/G_i$ where
$B_i \subset \C^n$ is an open ball and $G_i$ is a finite subgroup of $GL(n,\C)$.
\end{defi}
 
The concept of $V$-manifolds was introduced in~\cite{Satake56} and they have the same homological
properties over $\mathbb{Q}$ as manifolds. For instance, they admit a Poincar{\'e}
duality if they are compact and carry a pure Hodge structure if they are compact
and Kähler, see~\cite{Baily56}. They have been classified locally by
Prill~\cite{Prill67}. 

It is enough to consider the so-called
\emph{small subgroups} $G\subset\GL(n,\C)$, i.e.~without
rotations around hyperplanes other than the identity.

\begin{theo}\label{th_Prill}{\rm (\cite{Prill67}).}~Let $G_1$, $G_2$ be small
subgroups of $\GL(n,\C)$. Then $\C^n/G_1$ is isomorphic to $\C^n/G_2$ if and only
if $G_1$ and $G_2$ are conjugate subgroups. $\hfill \Box$
\end{theo}

We fix the notations when $G$ is abelian.

\begin{nothing}\label{notation_action}
For $\bd: = {}^t(d_1 \ldots d_r)$ we denote 
$\mu_{\bd} := \mu_{d_1} \times \cdots \times \mu_{d_r}$ a finite
abelian group written as a product of finite cyclic groups, that is, $\mu_{d_i}$
is the cyclic group of $d_i$-th roots of unity in $\C$. Consider a matrix of weight
vectors 
$$
A := (a_{ij})_{i,j} = [\ba_1 \, | \, \cdots \, | \, \ba_n ] \in\Mat (r
\times n, \Z),\quad
\ba_j:={}^t (a_{1 j}\dots a_{r j})\in\Mat(r\times 1,\Z),
$$ 
and the action
\begin{equation}\label{action_XdA}
\begin{array}{cr}
( \mu_{d_1} \times \cdots \times \mu_{d_r} ) \times \C^n  \longrightarrow  \C^n,&\bxi_\bd := (\xi_{d_1}, \ldots, \xi_{d_r}),
\\[0.15cm]
\big( \bxi_{\bd} , \mathbf{x} \big)  \mapsto  (\xi_{d_1}^{a_{11}} \cdot\ldots\cdot
\xi_{d_r}^{a_{r1}}\, x_1,\, \ldots\, , \xi_{d_1}^{a_{1n}}\cdot\ldots\cdot
\xi_{d_r}^{a_{rn}}\, x_n ),&
\mathbf{x} := (x_1,\ldots,x_n).
\end{array}
\end{equation}
Note that the $i$-th row of the matrix $A$ can be considered modulo $d_i$. The
set of all orbits $\C^n / G$ is called ({\em cyclic}) {\em quotient space of
type $(\bd;A)$} and it is denoted by
$$
  X(\bd; A) := X \left( \begin{array}{c|ccc} d_1 & a_{11} & \cdots & a_{1n}\\
\vdots & \vdots & \ddots & \vdots \\ d_r & a_{r1} & \cdots & a_{rn} \end{array}
\right).
$$
The orbit of an element $\mathbf{x}\in \C^n$ under this action is denoted
by $[\mathbf{x}]_{(\bd; A)}$ and the subindex is omitted if no ambiguity
seems likely to arise. Using multi-index notation
the action takes the simple form
\begin{eqnarray*}
\mu_\bd \times \C^n & \longrightarrow & \C^n, \\
(\bxi_\bd, \mathbf{x}) & \mapsto & \bxi_\bd\cdot\mathbf{x}:=(\bxi_\bd^{\ba_1}\, x_1, \ldots, \bxi_{\bd}^{\ba_n}\, x_n).
\end{eqnarray*}
\end{nothing}

The quotient of $\C^n$ by a finite abelian group
is always isomorphic to a quotient space of type $(\bd;A)$, see~\cite{kjj-qcw} for a proof of this classic result.
Different types $(\bd;A)$ can give rise to isomorphic quotient spaces.

\begin{ex}\label{Ex_quo_dim1} When $n=1$ all spaces $X(\bd;A)$ are
isomorphic to~$\C$. It is clear that we can assume that $\gcd(d_i, a_{i})=1$.
If $r=1$, the map $[x] \mapsto x^{d_1}$ gives an isomorphism between $X(d_1; a_{1})$ and
$\C$. 

Let us consider the case $r=2$. Note that
$\C/(\mu_{d_1} \times \mu_{d_2})$ equals $(\C/\mu_{d_1})/ \mu_{d_2}$.
Using the previous isomorphism, it is isomorphic to $X(d_2, d_1 a_2)$,
which is again isomorphic to $\C$. By induction, we obtain the result for any~$r$.
\end{ex}

The following lemma states some moves that leave unchanged the
isomorphism type of $X(\bd;A)$.

\begin{lemma}\label{movs-dA}
The following operations do not change the isomorphism type of $X(\bd;A)$.
\begin{enumerate}[\rm(1)]
\item Permutation $\sigma$ of columns of $A$, $[(x_1,\dots,x_n)]\mapsto[(x_{\sigma(1)},\dots,x_{\sigma(n)})]$.
\item Permutation of rows of $(\bd;A)$, $[\mathbf{x}]\mapsto[\mathbf{x}]$.
\item Multiplication of a row of $(\bd;A)$ by a positive integer, $[\mathbf{x}]\mapsto[\mathbf{x}]$.
\item Multiplication of a row of $A$ by an integer coprime with the
corresponding row in $\bd$, $[\mathbf{x}]\mapsto[\mathbf{x}]$.
\item Replace $a_{i j}$ by $a_{i j}+k d_j$, $[\mathbf{x}]\mapsto[\mathbf{x}]$.
\item If $e$ is coprime with $a_{1, n}$ and divides $d_1$ and $a_{1,j}$, $1\leq j<n$, then replace,
$a_{i,n}\mapsto e a_{i,n}$, $[(x_1,\dots,x_n)]\mapsto[(x_1,\dots,x_n^e)]$.
\item If $d_r=1$ then
eliminate the last row, $[\mathbf{x}]\mapsto[\mathbf{x}]$.
\end{enumerate}
\end{lemma}

Using Lemma~\ref{movs-dA}
we can prove the following lemma which restricts
the number of possible factors of the abelian group
in terms of the dimension.

\begin{lemma}\label{upper_triangular}
The space $X(\bd;A) = \C^n / \mu_d$ can always be represented by an upper
triangular matrix of dimension $(n-1) \times n$. More precisely, there exist a
vector ${\bf e} = (e_1,\ldots, e_{n-1})$, a~matrix $B = (b_{i,j})_{i,j}$, and an
isomorphism $[(x_1,\ldots,x_n)] \mapsto [(x_1,\ldots, x_n^{k})]$ for some
$k\in\N$ such that 
$$
X({\bf d}; A) \cong \left(\begin{array}{c|cccc}
e_1 & b_{1,1} & \cdots & b_{1,n-1} & b_{1,n} \\
\vdots & \vdots & \ddots & \vdots & \vdots\\
e_{n-1} & 0 & \cdots & b_{n-1,n-1} & b_{n-1,n}
\end{array}\right) = X({\bf e};B). 
$$
\end{lemma}

\begin{remark}
For $n=2$ it is enough to consider cyclic quotients. Nevertheless, in order
to avoid cumbersome statements, we will allow if necessary quotients
of non-cyclic groups.
\end{remark}

As we have already used, if an action is not free on $(\C^{*})^n$
we can factor the group by the kernel of the action and
the isomorphism type does not change.
With all these hypotheses we can define normalized types.

\begin{defi}\label{def_normalized_XdA}
The type $(\bd;A)$ is said to be {\em normalized} if the action is free on $(\C^{*})^n$
and~$\mu_\bd$ is small as subgroup of $GL(n,\C)$.
By abuse of language we often say the space $X(\bd;A)$ is written in a normalized
form when we mean the type $(\bd;A)$ is normalized.
\end{defi}

\begin{prop}
The space $X(\bd;A)$ is written in a normalized form if and only if the stabilizer
subgroup of $P$ is trivial for all~$P \in \C^n$ with exactly $n-1$ coordinates
different from zero.

In the cyclic case the stabilizer of a point as above (with exactly $\,n-1$
coordinates different from zero) has order $\gcd(d, a_1, \ldots, \widehat{a}_i,
\ldots, a_n)$.
\end{prop}

Using Lemma~\ref{movs-dA} it is possible
to convert general types~$(\bd;A)$ into their normalized form.
Theorem~\ref{th_Prill} allows one to decide whether two quotient spaces are
isomorphic. In particular, one can use this result to compute the singular points
of the space $X(\bd;A)$. If $n=2$, then a normalized type is always cyclic.

\begin{defi}\label{defiindex}
The \emph{index} of a quotient $X(\bd;A)$ of $\C^2$ equals $d$
for $X(\bd;A)\cong X(d;a,b)$ normalized. 
\end{defi}

In Example~\ref{Ex_quo_dim1} we have explained the previous normalization process in dimension one.
The two and three-dimensional cases are treated in the following examples. 


\begin{ex}\label{X2} Following Lemma~\ref{upper_triangular}, all
quotient spaces for $n=2$ are cyclic. The space $X(d;a,b)$ is written in a
normalized form if and only if $\gcd(d,a) = \gcd(d,b) = 1$. If this is not the
case, one uses the isomorphism\footnote{The notation $(i_1,\ldots,i_k) =
\gcd(i_1,\ldots,i_k)$ is used in case of complicated or long formulas.} (assuming
$\gcd(d,a,b)=1$)
$$
\begin{array}{rcl}
X(d; a,b)  & \longrightarrow & X \left( \frac{d}{(d,a)(d,b)}; \frac{a}{(d,a)},
\frac{b}{(d,b)} \right), \\[0.3cm]
\big[ (x,y) \big] & \mapsto & \big[ (x^{(d,b)},y^{(d,a)}) \big]
\end{array}
$$
to convert it into a normalized one, see also Lemma~\ref{movs-dA}.
\end{ex}

\begin{ex}\label{Ex_dim3} The quotient space $X(d;a,b,c)$ is
written in a normalized form if and only if $\gcd(d,a,b) = \gcd(d,a,c) =
\gcd(d,b,c) = 1$. As above, isomorphisms of the form $[(x,y,z)] \mapsto
[(x,y,z^k)]$ can be used to convert types $(d;a,b,c)$ into their normalized
form.
\end{ex}

\begin{remark}\label{rem-simpli}
Let us show how to convert a space of type
$
\left(
\begin{smallmatrix}
p\\
q
\end{smallmatrix}
\right|
\left.
\begin{smallmatrix}
a&b\\    
c&d                                                      
\end{smallmatrix}
\right)
$ into its cyclic form.
By suitable multiplications of the rows, we can assume $p=q=r$:
$X
\left(
\begin{smallmatrix}
r\\
r
\end{smallmatrix}
\right|
\left.
\begin{smallmatrix}
 a&  b\\    
 c&  d                                                      
\end{smallmatrix}
\right).
$
For the second step we add a third row by adding the first row multiplied
by $\alpha$ and the second row multiplied by $\beta$, where $\alpha  a +\beta  c=m$
and $m:=\gcd( a, c)$ (note that $\gcd(\alpha,\beta)=1$):
$$
X
\left(
\begin{matrix}
r\\
r\\
r
\end{matrix}
\right|
\left.
\begin{matrix}
a&  b\\    
c&  d \\
m&\alpha b +\beta d                                                     
\end{matrix}
\right)=
X
\left(
\begin{matrix}
r\\
r\\
r
\end{matrix}
\right|
\left.
\begin{matrix}
0& -\beta\frac{a d -b c}{m}\\    
0&  \alpha\frac{a d -b c}{m} \\
m&\alpha b +\beta d                                                     
\end{matrix}
\right)=
X
\left(
\begin{matrix}
r\\
r
\end{matrix}
\right|
\left.
\begin{matrix}
0&\frac{a d -b c}{m} \\
m&\alpha b +\beta  d                                                     
\end{matrix}
\right).
$$
Let $t:=\gcd(r,\frac{a d -b c}{m})$. Then, our space
is of type $(r;m, (\alpha  b +\beta  d) \frac{r}{t})$ and normalization
follows by taking $\gcd$'s. The isomorphism is
$[(x,y)]\mapsto[(x,y^{\frac{r}{t}})]_{(r;m, (\alpha  b +\beta  d) \frac{r}{t})}$.
\end{remark}

In~\cite{Fujiki74} the author computes resolutions of cyclic quotient
singularities. In the $2$-dimensional case, the resolution process is due to Jung and Hirzebruch,
see~\cite{hnk}.

\section{Weighted Projective Spaces}\label{weighted_projective_space}


The main reference that has been used in this section is \cite{Dolgachev82}.
Here we concentrate our attention on the analytic structure.

Let $\w:=(q_0,\ldots,q_n)$ be a weight vector, that is, a finite set of coprime positive
integers. 
There is a natural action of the multiplicative group $\C^{*}$ on
$\C^{n+1}\setminus\{0\}$ given by
$$
  (x_0,\ldots,x_n) \longmapsto (t^{q_0} x_0,\ldots,t^{q_n} x_n).
$$

The set of orbits $\frac{\C^{n+1}\setminus\{0\}}{\C^{*}}$ 
under this action is denoted by $\bP^n_\w$ (or $\bP^n(\w)$ in case of complicated weight vectors) 
and it is called the {\em weighted projective space} of type $\w$. 
The class of a nonzero element $(x_0,\ldots,x_n)\in \C^{n+1}$ 
is denoted by $[x_0:\ldots:x_n]_\w$ and the weight vector is omitted if no ambiguity seems likely to arise.
When $(q_0,\ldots,q_n)=(1,\ldots,1)$ one obtains the usual projective space
and the weight vector is always omitted. For $\mathbf{x}\in\C^{n+1}\setminus\{0\}$,
the closure of $[\mathbf{x}]_\w$ in $\C^{n+1}$ is obtained by adding the origin and it
is an algebraic curve.

\begin{nothing}\label{analytic_struc_Pkw}\textbf{Analytic structure.} Consider the decomposition
$
\bP^n_\w = U_0 \cup \cdots \cup U_n,
$
where $U_i$ is the open set consisting of all elements $[x_0:\ldots:x_n]_\w$
with $x_i\neq 0$. The map
$$
  \widetilde{\psi}_0: \C^n \longrightarrow U_0,\quad
\widetilde{\psi}_0(x_1,\cdots,x_n):= [1:x_1:\ldots:x_n]_\w
$$
defines an isomorphism $\psi_0$ if we replace $\C^n$ by
$X(q_0;\, q_1,\ldots,q_n)$.
Analogously, $X(q_i;\,q_0,\ldots,\widehat{q}_i,\ldots,q_n) \cong U_i$
under the obvious analytic map.
\end{nothing}

\begin{prop}[\cite{kjj-qcw}]\label{propPw}
Let $d_i := \gcd (q_0,\ldots,\widehat{q}_i,\ldots,q_n)$,
 $e_i:= d_0\cdot\ldots\cdot\widehat{d}_i\cdot\ldots\cdot d_n$
and $p_i:=\frac{q_i}{e_i}$. The following map is an isomorphism:
$$
\begin{array}{rcl}
\bP^n \big(q_0,\ldots,q_n\big) & \longrightarrow & \bP^n(p_0,\dots,p_n)
\\
\,[x_0:\ldots:x_n] & \mapsto &
\big[\,x_0^{d_0}:\ldots:x_n^{d_n}\,\big].
\end{array}
$$
\end{prop}

\begin{remark}
Note that, due to the preceding proposition, one can always assume the weight
vector satisfies $\gcd(q_0,\ldots,\widehat{q}_i,\ldots,q_n)=1$, for
$i=0,\ldots,n$. In particular, $\bP^1{(q_0,q_1)} \cong \bP^1$ and for $n=2$ we can
take $(q_0,q_1,q_2)$ relatively prime numbers. In higher dimension the situation
is a bit more complicated.
\end{remark}

\section{Abstract and Embedded $\Q$-Resolutions}

Classically an embedded resolution of $\{f=0\} \subset \C^n$ is a proper map
$\pi: X \to (\C^n,0)$ from a smooth variety $X$ satisfying, among other
conditions, that $\pi^{-1}(\{f=0\})$ is a normal crossing divisor. To weaken the
condition on the preimage of the singularity we allow the new ambient space $X$
to contain abelian quotient singularities and the divisor $\pi^{-1}(\{f=0\})$ to
have \emph{normal crossings} over this kind of varieties. This notion of normal
crossing divisor on $V$-manifolds was first introduced by Steenbrink
in~\cite{Steenbrink77}.

\begin{defi}
Let $X$ be a $V$-manifold with abelian quotient singularities. A hypersurface
$D$ on $X$ is said to be with {\em $\mathbb{Q}$-normal crossings} if it is
locally isomorphic to the quotient of a union of coordinate hyperplanes under a group
action of type $(\bd;A)$. That is, given $x \in X$, there is an isomorphism of
germs $(X,x) \simeq (X(\bd;A), [0])$ such that $(D,x) \subset (X,x)$ is identified
under this morphism with a germ of the form
$$
\big( \{ [\mathbf{x}] \in X(\bd;A) \mid x_1^{m_1} \cdot\ldots\cdot x_k^{m_k} = 0 \},
[(0,\ldots,0)] \big).
$$
\end{defi}

Let $M = \C^{n+1} / \mu_\bd$ be an abelian quotient space not necessarily cyclic
or written in normalized form.
Consider $H \subset M$ an analytic subvariety of codimension one.

\begin{defi}\label{Qresolution}
An {\em embedded $\Q$-resolution} of $(H,0) \subset (M,0)$ is a proper analytic
map~$\pi: X \to (M,0)$ such that:
\begin{enumerate}
\item $X$ is a $V$-manifold with abelian quotient singularities.
\item $\pi$ is an isomorphism over $X\setminus \pi^{-1}(\Sing(H))$.
\item $\pi^{-1}(H)$ is a hypersurface with $\mathbb{Q}$-normal crossings on $X$.
\end{enumerate}
\end{defi}

\begin{remark}
Let $f:(M,0) \to (\C,0)$ be a non-constant analytic function germ. Consider
$(H,0)$ the hypersurface defined by $f$ on $(M,0)$. Let $\pi: X \to (M,0)$ be an
embedded $\Q$-resolution of $(H,0) \subset (M,0)$. Then $\pi^{-1}(H) = (f\circ
\pi)^{-1}(0)$ is locally given by a function of the form
$
x_1^{m_1}  \cdot\ldots\cdot x_k^{m_k} : X(\bd;A) \rightarrow \C.
$ 
\end{remark}

In the same way we define abstract $\Q$-resolutions.

\begin{defi}
Let $(X,0)$ be a germ of singular point. An \emph{abstract good $\Q$-resolution}
is a proper birational morphism $\pi:\hat{X}\to(X,0)$ such that
$\hat{X}$ is a $V$-manifold with abelian quotient singularities,
$\pi$ is an isomorphism outside $\Sing(X)$, and
$\pi^{-1}(\Sing(X))$ is a $\mathbb{Q}$-normal crossing divisor.
\end{defi}

\begin{notation}\label{not-grafodual}
It is usual to encode normal crossing divisors by its \emph{dual complex}:
one vertex for each irreducible component, one edge for each intersection
of two irreducible components, one triangle for intersection of three
irreducible components and so on. It is particularly useful for normal
crossing divisors in surfaces where one deals with (weighted) graphs.

We explain how to encode $\mathbb{Q}$-normal crossings with a weighted
graph in the case of surfaces. We are interested in two cases:
the divisor $\pi^{*}(H)=(f\circ\pi)^*(0)$ for an embedded $\Q$-resolution
of a curve $H=f^*(0)$ and the exceptional divisor of
an abstract good $\Q$-resolution~$\pi$ 
of a normal surface. We associate to such divisors a weighted
graph $\Gamma$ as follows:
\begin{itemize}
\item The set $V_\Gamma$ of vertices of $\Gamma$ is the ordered set of irreducible components of $\pi^{*}(H)$
(for some arbitrary order).
It is decomposed in two subsets $V_\Gamma=V_\Gamma^0\coprod V_\Gamma^H$; the first subset corresponds
to exceptional components and the second to strict transforms (using arrow-ends). The set
$V_\Gamma^H$ is empty when the divisor is compact (e.g. when
$\pi$ is an abstract good $\Q$-resolution).

\item The set $E_\Gamma$ of edges of $\Gamma$ is in bijection with the double points of $\pi^{*}(H)$.

\item Each $E\in V_\Gamma^0$ is weighted by its genus $g_E$
(usually omitted if $g_E=0$. It is also weighted by
its self-intersection~$e_E\in\mathbb{Q}$, see Definition~\ref{definition_intersection_number} later.

\item When $\pi$ is an embedded $\Q$-resolution, each
$E\in V_\Gamma$ is weighted by $m_E$.
The multiplicity $m_E$ is defined as follows: given a generic point in $E$ one can choose
local analytic coordinates $(x_E,y_E)$ centered at this point such that
$y_E=0$ is a local equation of $E$ and $(f\circ\pi)(x_E,y_E)=y_E^{m_E}$.

\item For $E\in V_\Gamma$, let $\Sing^0(E)$ be the set of singular points of
$\hat{X}$ in $E$ which are not double points. Then, we associate
to $E$ the sequenceof normalized types $\{(d_P; a_P,b_P)\}_{P\in\Sing^0(E)}$,
where for $P\in\Sing^0(E)$, $E$ is the image
of $y=0$. Note that $d_P$ divides $m_E$.

\item If the double point $P_\gamma=E_1\cap E_2$, $E_1<E_2$, 
associated with $\gamma\in E_\Gamma$ is singular, we associate
to it a normalized type $(d; a,b)$, where $E_1$ is the image
of $x=0$ and $E_2$ is the image of $y=0$. Note that $d$ divides
$a m_{E_1}+b m_{E_2}$.
\end{itemize}
This notation is also useful for exceptional graphs of good $\Q$-resolutions.
\end{notation}

\section{Weighted Blow-ups}

Weighted blow-ups can be defined in any dimension, see~\cite{kjj-qcw}. In this
section, we restrict our attention to the case~$n=2$.

\begin{nothing}\label{blowup_dimN_smooth}\textbf{Classical blow-up of $\C^{2}$.}
We consider
$$
\widehat{\C}^{2} := \big\{ ((x,y),[u:v]) \in \C^{2} \times \bP^1 \mid
(x,y)\in \overline{[u:v]} \big\}.
$$
Then $\pi : \widehat{\C}^{2} \to \C^{2}$ is an isomorphism over
$\widehat{\C}^{2} \setminus \pi^{-1}(0)$. The {\em exceptional divisor} $E:=
\pi^{-1}(0)$ is identified with $\bP^{1}$. The space
$\widehat{\C}^{2} = U_0 \cup U_1$ can be covered with $2$
charts each of them isomorphic to $\C^2$. For instance, the following map
defines an isomorphism:
\begin{eqnarray*}
\C^{2} & \longrightarrow & U_0 = \{ u \neq 0 \} \ \subset \
\widehat{\C}^{2},\\
(x,y) \ & \mapsto & \big( (x, x y),[1: y] \big).
\end{eqnarray*}
\end{nothing}

\begin{nothing}\label{weighted_blowup_dimN_smooth}\textbf{Weighted
$(p,q)$-blow-up of $\C^{2}$.} 
Let~$\w = (p_0,p_1)$ be a weight
vector with coprime entries. As above, consider the space
$$
\widehat{\C}^{2}(\w) := \big\{ ((x,y),[u:v]_{\w}) \in \C^{2} \times
\bP^1_{\w} \mid (x,y) \in \overline{[u:v]}_{\w}
\big\}.
$$
covered~as
$
\widehat{\C}^2_{\w} = U_1 \cup U_2 = X(p;-1,q) \cup X(q;\,p,-1)
$
and the charts are given by
$$
\begin{array}{c|rcl}
\text{First chart} & X(p;-1,q) & \longrightarrow & U_1, \\[0.10cm]
& \,[(x,y)] & \mapsto & ((x^p,x^q y),[1:y]_{\w}). \\ \multicolumn{2}{c}{} \\
\text{Second chart} & X(q;\,p,-1) & \longrightarrow & U_2, \\[0.10cm]
& \,[(x,y)] & \mapsto & ((x y^p, y^q),[x:1]_{\w}).
\end{array}
$$
The exceptional divisor $E:=\pi_{\w}^{-1}(0)$ is isomorphic to $\bP^1_{\w}$ which
is in turn isomorphic to~$\bP^1$ under the map
$
[x:y]_{\w} \longmapsto [x^q:y^p]$.
The singular points of $\widehat{\C}^2_{\w}$ are cyclic quotient singularities
located at the exceptional divisor of indices $p$ and $q$. They actually coincide with the origins of
the two charts and are written in their normalized form.
\end{nothing}

Let us study now the weighted blow-ups of quotient spaces. The general computations were made
in~\cite{kjj-qcw} and we specialize here for dimension~$2$. We study 
the $(p,q)$-blow-up of $X(d;a,b)$ (in normalized
form) and for simplicity we start with the case $p=a$, $q=b$.

\begin{nothing}\label{blow-up2-sing-pq}
{\bf Blow-up of $X(d;p,q)$ with respect to $\w:=(p,q)$.} Let $X:=X(d; p,q)$ 
in a normalized form, i.e.~$\gcd(d,p) = \gcd(d,q) = 1$. Denote
$\pi := \pi_{\w,d}: \widehat{\C}^2_{\w,d} \to X(d;p,q)$ the weighted blow-up
at the origin with respect to $\w=(p,q)$. Following \cite{kjj-qcw},
we cover $\widehat{X(d;p,q)}_{\w,d}$ by two charts. The first
one is of type

$$X \left(\begin{array}{c|cc}
   p &-1 &q \\
   pd &p & pq-qp\\
  \end{array}\right)=
X \left(\begin{array}{c|cc}
   p &-1 &q \\
   d &1 & 0\\
  \end{array}\right) \stackrel{x\mapsto x^d}{\cong} X(p;-d,q).
$$

Hence 
$
\widehat{\C}^2_{\w,d} = U_1 \cup U_2 = X(p;-d,q) \cup X(q;p,-d)
$
and the charts are given by
\begin{equation}\label{eq-cartaspq}
\begin{array}{c|rcl}
\text{First chart} & X(p;-d,q) & \longrightarrow & U_1, \\[0.10cm]
& \,\big[ (x^d,y) \big] & \mapsto & \big( [(x^p,x^q y)]_{d},[1:y]_{\w} \big). \\
\multicolumn{2}{c}{} \\
\text{Second chart} & X(q;\,p,-d) & \longrightarrow & U_2, \\[0.10cm]
& \big[ (x,y^d) \big] & \mapsto & \big( [(x y^p, y^q)]_{d},[x:1]_{\w} \big).
\end{array}
\end{equation}
As above the exceptional divisor $E:=\pi_{\w}^{-1}(0)$ is identified with
$\bP^1_{\w}$ which is isomorphic to $\bP^1$ under the map
$
[x:y]_{\w} \longmapsto [x^q:y^p]
$.
The singular points of $\widehat{\C}^2_{\w,d}$ are cyclic quotient singularities,
they coincide with the origins of the two charts and are written in their
normalized form.
\end{nothing}

\begin{figure}
\centering
\subfigure[Blow-up]{\includegraphics{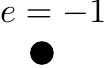}}\qquad\qquad
\subfigure[$\w$-blow-up]{\includegraphics{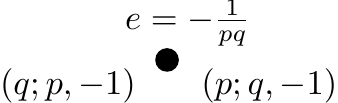}}\qquad
\subfigure[$\w$-blow-up of $X(d;p,q)$]{\includegraphics{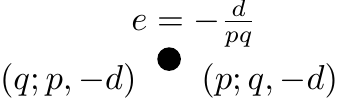}}%
\caption{Graphs of the blow-ups.}
\label{fig-blowup1}
\end{figure}

\begin{defi}\label{def-multi}
Let $\pi: \widehat{X(\bd;A)}(\w) \to X(\bd;A)$ be the $\w$-blow-up. Then the {\em total transform}
$\pi^{*}(H)$, $H:=f^{-1}(0)$, for some $f:X(\bd;A)\to\C$ holomorphic, decomposes as
$$
  \pi^{*}(H) = \widehat{H} + m E,
$$
where $E := \pi^{-1}(0)$ is the {\em exceptional divisor} of $\pi$, $\widehat{H}
:= \overline{\pi^{-1}(H\setminus L)}$ is the {\em strict transform} of $H$, and
$m$ is the {\em multiplicity} of $E$ at a smooth point.
\end{defi}

\begin{remark}\label{rem-multi}
In order to compute multiplicities when looking at \emph{multicharts} (for quotient spaces)
we must be careful with the expressions in coordinates in case the space is represented by a non-normalized type. For instance, if a divisor is locally given by the function
$
x^{md}:X\left(
\begin{smallmatrix}
p\\
d
\end{smallmatrix}
\right|
\left.
\begin{smallmatrix}
-1&q\\    
1&0                                                      
\end{smallmatrix}
\right)\rightarrow \C
$, its multiplicity is $m$.
\end{remark}

\begin{ex}\label{example_2}
Assume $\gcd(p,q)=1$ and $p<q$. Let $f = (x^p + y^q)(x^q + y^p)$ and consider
${\bf C}_1 = \{ x^p + y^q = 0 \}$ and ${\bf C}_2 = \{ x^q + y^p = 0 \}$ the two
irreducible components of $\{ f=0 \}$.

Let $\pi_{(q,p)}: \widehat{\C}^2_{(q,p)} \to \C^2$ be the $(q,p)$-weighted
blow-up at the origin. The new space has two singular points of type $(q;-1,p)$
and $(p;q,-1)$ located at the exceptional divisor $\mathcal{E}_1$. The local
equation of the total transform in the first chart is given by the function
$$
  x^{p(p+q)}(1+y^q)(x^{q^2-p^2}+y^p) : X(q;-1,p) \longrightarrow \C,
$$
where $x=0$ is the equation of the exceptional divisor and the other factors
correspond to the strict transform of ${\bf C}_1$ and ${\bf C}_2$ (denoted again
by the same symbol). Due to the cyclic action, $y^q+1=0$ produces only one branch.

Hence $\mathcal{E}_1$ has multiplicity $p(p+q)$; it intersects transversally
${\bf C}_1$ at a smooth point while it intersects ${\bf C}_2$ at a singular
point (the origin of the first chart) without normal crossings.

\begin{figure}[h t]
\centering\includegraphics{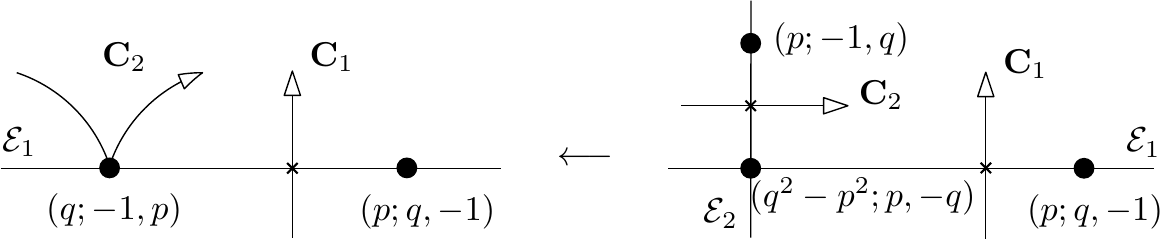}
\caption{Embedded $\Q$-resolution of $\{ (x^p+y^q)(x^q+y^p) = 0 \} \subset \C^2$.}
\label{fig-blowup_curve2}
\end{figure}

Note that $X(q;-1,p)=X(q;p,-p^2)=X(q;p,q^2-p^2)$ and we can apply~\ref{blow-up2-sing-pq}.
Let us consider $\pi_{(p,q^2-p^2),q}$ the weighted blow-up at the origin of
$X(q;-1,p)$ with respect to $(p,q^2-p^2)$,
$$
\pi_{(p,q^2-p^2),q} :\, \widehat{\C}^2_{(p,q^2-p^2),q} \longrightarrow 
X(q; p,q^2-p^2) = X(q;-1,p).
$$
The new space has two singular points of type $(p;-q,q^2-p^2) = (p;-1,q)$ and
$(q^2-p^2; p, -q)$. In the first chart, the local equation of the total
transform of $x^{p(p+q)}(x^{q^2-p^2}+y^p)$ is given by the function
$$
  x^{p (p+q)}(1+y^p): \, X(p;-1,q) \longrightarrow \C.
$$
Thus the new exceptional divisor $\mathcal{E}_2$ has multiplicity $p (p+q)$ and
intersects transversally the strict transform of ${\bf C}_2$ at a smooth point.
Hence $\pi_{(p,q^2-p^2),q} \circ \pi_{(q,p)}$ is an embedded
$\Q$-resolution of $\{f=0\} \subset \C^2$ where all quotient spaces are written
in a normalized form. Figure~\ref{fig-blowup_curve2} illustrates the whole process
and Figure~\ref{fig-grafo-res1} shows the dual graph.

\begin{figure}[h t]
\centering\includegraphics{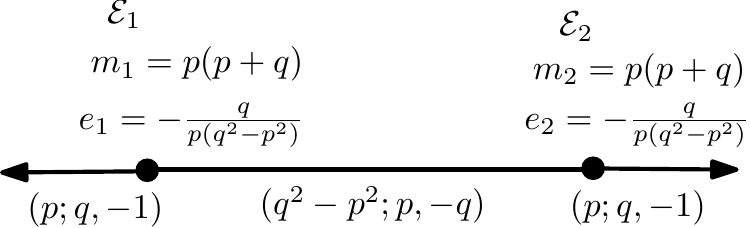}
\caption{Dual graph of the embedded $\Q$-resolution of $\{ (x^p+y^q)(x^q+y^p) = 0 \} \subset \C^2$.}
\label{fig-grafo-res1}
\end{figure}
\end{ex}

We consider now the general case.

\begin{nothing}\label{blow-up2-sing-ab}
{\bf Blow-up of $X(d;a,b)$ with respect to $\w:=(p,q)$.} Let
$X=X(d;a,b)$ assumed to be normalized. Let
$$\pi: = \pi_{(d;a,b),\w}: \, \widehat{X(d;a,b)}_{\w}
\longrightarrow X(d;a,b)
$$ 
be the weighted blow-up at the origin of $X(d;a,b)$
with respect to $\w = (p,q)$. Then, $\widehat{X(d;a,b)}_{\w}$ is covered by
$$
\widehat{U}_1 \cup \widehat{U}_2 = X \left( \begin{array}{c|cc} p & -1 & q \\ pd
& a & pb - qa \end{array} \right) \cup X \left( \begin{array}{c|cc} q & p & -1
\\ qd & qa-pb & b \end{array} \right)
$$
and the charts are given by 
$$
\begin{array}{c|c}
\text{First chart} & X \left( \begin{array}{c|cc} p & -1 & q \\ pd & a & pb - qa
\end{array} \right) \ \longrightarrow \ \widehat{U}_1, \\[0.5cm] & \,\big[ (x,y)
\big] \mapsto \big[ ((x^p,x^q y),[1:y]_{\w}) \big]_{(d;a,b)}. \\
\multicolumn{2}{c}{} \\
\text{Second chart} & X \left( \begin{array}{c|cc} q & p & -1 \\ qd & qa-pb & b
\end{array} \right) \ \longrightarrow \ \widehat{U}_2, \\[0.5cm]
& \big[ (x,y) \big] \mapsto \big[ ((x y^p, y^q),[x:1]_{\w}) \big]_{(d;a,b)}.
\end{array}
$$
The exceptional divisor $E = \pi_{(d;a,b),\w}^{-1}(0)$ is identified with the
quotient space $\bP^1_{\w}(d;a,b) := \bP^1_{\w}/\mu_{d}$ which is isomorphic to
$\bP^1$ under the map
$$
\begin{array}{rcl}
\bP^1_{\w}(d;a,b) & \longrightarrow & \bP^1, \\[0.1cm]
\,[x:y]_{\w,(d;a,b)} & \mapsto & [x^{dq/e}: y^{dp/e}],
\end{array}
$$
where $e: = \gcd(d,pb-qa)$. Again the singular points are cyclic and
correspond to the origins.

Let us apply Remark~\ref{rem-simpli} to the preceding charts.
Assume the type $(d;a,b)$ is normalized. To normalize these quotient spaces, note
that $e =\gcd(d,pb-qa)= \gcd(d, -q+\beta p b) =\gcd(pd, -q+\beta p b)=\gcd(qd,p-q a\mu)$, where
$\beta a\equiv \mu b\equiv 1\mod d$.

Then another expressions for the two charts are given below.
\begin{equation}\label{blowup_normalized_dim2}
\begin{array}{c|c}
\text{First chart} & X \left( \displaystyle\frac{pd}{e}; 1, \frac{-q+\beta p
b}{e} \right)  \ \longrightarrow \ \widehat{U}_1, \\[0.5cm] & \,\big[ (x^e,y)
\big] \mapsto \big[ ((x^p,x^q y),[1:y]_{\w}) \big]_{(d;a,b)}. \\
\multicolumn{2}{c}{} \\
\text{Second chart} & X \left( \displaystyle\frac{qd}{e}; \frac{-p+\mu qa}{e}, 1
\right) \ \longrightarrow \ \widehat{U}_2, \\[0.5cm] & \hspace{0.15cm} \big[
(x,y^e) \big] \mapsto \big[ ((x y^p, y^q),[x:1]_{\w}) \big]_{(d;a,b)}.
\end{array}
\end{equation}
Both quotient spaces are now written in their normalized
form. The equation of the charts will be useful to compute multiplicities, see Definition~\ref{def-multi}
and Remark~\ref{rem-multi}.
\end{nothing}

\begin{ex}\label{example_3}
Assume $\gcd(p,q) = \gcd(r,s) = 1$ and $\frac{p}{q} < \frac{r}{s}$. Let $f =
(x^p + y^q) (x^r + y^s)$ and consider ${\bf C}_1 = \{ x^p + y^q = 0 \}$ and
${\bf C}_2 = \{ x^r + y^s = 0 \}$. Working as in Example~\ref{example_2}, one
obtains Figure~\ref{fig-blowup-curve3} representing an embedded $\Q$-resolution of
$\{f=0\} \subset \C^2$. We start with a $(q,p)$-blow-up and we continue
with an $(s,q r-p s)$-blow-up over a point of type $(q;-1,p)$.

\begin{figure}[h t]
\centering\includegraphics{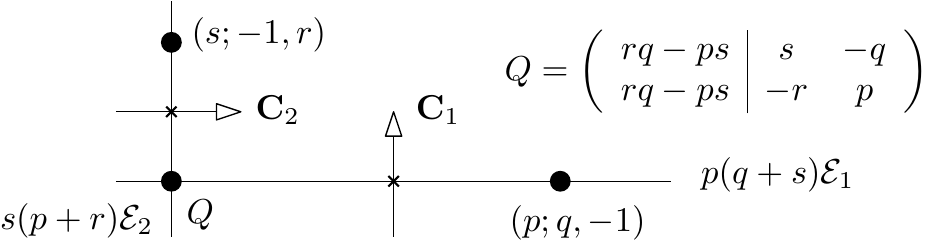}
\caption{Embedded $\Q$-resolution of $\{ (x^p+y^q)(x^r+y^s) = 0 \} \subset
\C^2$.}
\label{fig-blowup-curve3}
\end{figure}

The point $Q$ is also of type $(rq-ps;ar+bs,-1)$ where $ap+bq=1$.
In fact, it is in normalized form, since $\gcd(rq-ps,ar+bs)=1$.
After writing the quotient spaces in their normalized form one checks that this
resolution coincides with the one given in Example~\ref{example_2} assuming
$r=q$ and $s=p$.
The dual graph is shown in Figure~\ref{fig-grafo-res2}.

\begin{figure}[h t]
\centering\includegraphics{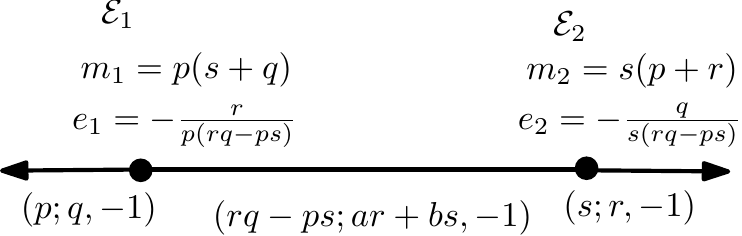}
\caption{Dual graph of the embedded $\Q$-resolution of $\{ (x^p+y^q)(x^r+y^s) = 0 \} \subset \C^2$,
$ap+bq=1$.}
\label{fig-grafo-res2}
\end{figure}
\end{ex}

\begin{nothing}\label{puiseux}\textbf{Puiseux expansion.}
Let us study the behavior of Puiseux pairs
under weighted blow-ups. Let ${\bf C} = \{ f = 0 \} \subset \C^2$ be the
irreducible plane curve given by
$$
\prod_{j=1}^d \big[ -y + (a_{11}x^{\frac{p_1}{q}} + \cdots + a_{k1}
x^{\frac{p_k}{q}}) + (a_{12} x^{\frac{r_1}{q s}} + \cdots + a_{l2}
x^{\frac{r_l}{q s}}) + \cdots \big],
$$
where $p_1 < \cdots < p_k$, $r_1 < \cdots < r_l$, $\frac{p_k}{q} <
\frac{r_1}{q s}$, $\gcd(p_1,q)=\gcd(r_1,s)=1$, and $q,s>1$
(after a  change of variables we may assume the first term
has non-integer exponent). 

Let $\pi_{(q,p_1)}: \widehat{\C}^2_{(q,p_1)} \to \C^2$ be the $(q,p_1)$-weighted
blow-up at the origin. In the first chart, that is, after performing the
substitution $(x,y) \mapsto (x^q,x^{p_1}y)$,
one obtains the following equation for the total transform
\begin{equation*}
\begin{split}
x^{p_1 d} \cdot \prod_{j=1}^d \big[ -y \, + \, & (a_{11} + a_{21} x^{p_2-p_1} +
\cdots + a_{k1} x^{p_k-p_1}) + \\ & (a_{12} x^{\frac{r_1 - p_1 s}{s}} + \cdots
+ a_{l2} x^{\frac{r_l - p_1 s}{s}}) + \cdots \big] = 0.
\end{split}
\end{equation*}
At first sight the exceptional divisor and the strict transform intersect at $d$
different smooth points. However, since $a_{1j}^q$ does not depend on $j$ by
conjugation, all of them are the same.

After change of coordinates 
$
y \mapsto y + (a_{11} +a_{21} x^{p_2-p_1} + \cdots + a_{k1} x^{p_k-p_1}),
$
the local equation of the total transform $\pi^{*}_{(q,p_1)}({\bf C})$ at this
point is
$$
  x^{p_1 d} \cdot \prod_{j=1}^{d/q} \big[ -y + (a_{12} x^{\frac{r_1  - p_1
s}{s}} + \cdots + a_{l2} x^{\frac{r_l  - p_1 s}{s}}) + \cdots \big] = 0.
$$
This proves that in the irreducible case, only a weighted blow-up is needed for
each Puiseux pair in order to compute an embedded $\Q$-resolution, and the weight
is determined by the Puiseux pairs. Moreover, the
embedded $\Q$-resolution obtained is as in
Figure~\ref{figure_puiseux_expansion}.

\begin{figure}[h t]
\centering
\includegraphics{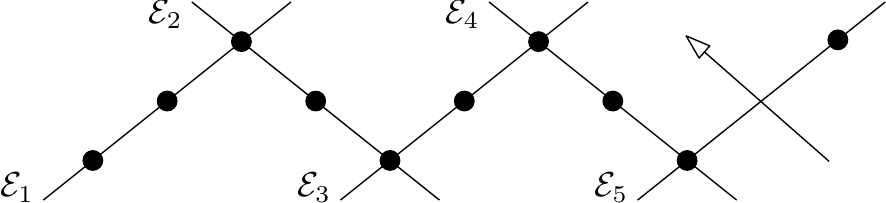}
\caption{Embedded $\Q$-resolution of an irreducible plane
curve.}\label{figure_puiseux_expansion}
\end{figure}

In the reducible case, one has to consider the weighted blow-ups
associated with the Puiseux pairs
of each irreducible component and add also weighted blow-ups associated
with the contact exponents for each pair of 
branches. There is another longer way to get this $\Q$-resolution:
perform a standard embedded resolution and contract any exceptional component having
at most two singular points in the divisor, cf.~\cite{Veys97}.
\end{nothing}

\begin{ex}\label{ex-hirzebruch1}
Let us consider $X:=\bP^2_\w$, for $\w=(p,q,r)$. We recall that $P:=[0:1:0]_\w$ is
a singular point of type $(q;p,r)$. We are going to perform the $(p,r)$-blow-up at this point.
The new surface $\hat{X}_P$ admits a map onto 
$\pi:\hat{X}_P\to\bP^1_{(p,r)}\cong\bP^1$ with rational fibers. This surface
has (at most) four singular points; two of them come from $X$ and they are of type
$(p;q,r)$, $Q:=[1:0:0]_\w$, and $(r;p,q)$, $R:=[0:0:1]_\w$. 
The other two points are in the exceptional divisor
$E$ and they are of type $(p;-q,r)$ and $(r;p,-q)$; 
the singular points which are quotient by $\mu_p$ are in the
same fiber for $\pi$
and the same happens for $\mu_r$. The map has two relevant sections, $E$ and the transform of 
$y=0$.
\end{ex}

\section{Cartier and Weil $\mathbb{Q}$-Divisors on $V$-Manifolds}

We recall the definitions
of Cartier and Weil divisors.
Let $X$ be an irreducible normal complex analytic variety. Denote $\cO_X$
the structure sheaf of $X$ and $\mathcal{K}_X$ the sheaf of total quotient rings
of $\cO_X$. Denote by $\mathcal{K}_X^{*}$ the (multiplicative) sheaf of
invertible elements in $\mathcal{K}_X$. Similarly $\cO_X^{*}$ is the sheaf of
invertible elements in~$\cO_X$. 
Note that an irreducible subvariety
$V$ corresponds to a prime ideal in the ring of sections of any local complex
model space meeting $V$.

\begin{defi}
A {\em Cartier divisor} on $X$ is a global section of the sheaf
$\mathcal{K}_X^{*}/\cO_X^{*}$ and it can be represented by giving an open covering $\{U_i\}_{i\in
I}$ of $X$ and, for all $i\in I$, an element $f_i \in \Gamma (U_i,
\mathcal{K}_X^{*})$ such that
$$
\frac{f_i}{f_j} \in \Gamma(U_i \cap U_j, \cO_X^{*}), \quad \forall i,j \in I.
$$
Two systems $\{(U_i,f_i)\}_{i\in I}$, $\{(V_j,g_j)\}_{j\in J}$ represent the
same Cartier divisor if and only if on $U_i \cap V_j$, $f_i$ and $g_j$ differ by
a multiplicative factor in $\cO_X(U_i\cap V_j)^{*}$.
The abelian group of Cartier divisors on $X$ is denoted by $\CaDiv(X)$. If $D :=
\{(U_i,f_i)\}_{i\in I}$ and $E := \{(V_j,g_j)\}_{j\in J}$ then $D + E = \{(U_i
\cap V_j, f_i g_j)\}_{i\in I, j\in J}$. 
\end{defi}

The functions $f_i$ above are called {\em local equations} of the divisor on
$U_i$. A Cartier divisor on $X$ is {\em effective} if it can be represented by
$\{(U_i,f_i)\}_i$ with all local equations $f_i \in \Gamma(U_i,\cO_{X})$.

Any global section~$f \in \Gamma(X,\mathcal{K}^{*}_X)$ determines a {\em
principal} Cartier divisor $(f)_X := \{(X,f)\}$ by taking all local equations
equal to $f$. That is, a Cartier divisor is principal if it is in the image of
the natural map $\Gamma(X,\mathcal{K}_X^{*}) \to
\Gamma(X,\mathcal{K}_X^{*}/\cO_X^{*})$. Two Cartier divisors $D$ and $E$ are {\em
linearly equivalent}, denoted by $D\sim E$, if they differ by a principal
divisor. The {\em Picard group} $\Pic(X)$ denotes the group of linear
equivalence classes of Cartier divisors.

The {\em support} of a Cartier divisor $D$, 
denoted by $\Supp(D)$ or $|D|$, is the subset of $X$ consisting of all points $x$ 
such that a local equation for $D$ is not in $\cO_{X,x}^{*}$. The support of $D$ is a closed subset of $X$.

\begin{defi}\label{defi_Weil}
A {\em Weil divisor} on $X$ is a locally finite linear combination with integral
coefficients of irreducible subvarieties of codimension one. The abelian group
of Weil divisors on $X$ is denoted by $\WeDiv(X)$. If all coefficients appearing
in the sum are non-negative, the Weil divisor is called {\em effective}.
\end{defi}

Given a Cartier divisor, using 
the notion of order of a divisor along an irreducible subvariety of codimension
one, there is a Weil divisor associated with it.
Let $V \subset X$ be an
irreducible subvariety of codimension one. It corresponds to a prime ideal in
the ring of sections of any local complex model space meeting~$V$. The {\em
local ring of $X$ along $V$}, denoted by $\cO_{X,V}$, is the localization of such
ring of sections at the corresponding prime ideal; it is a one-dimensional local
domain.

For a given $f \in \cO_{X,V}$ define $\ord_V(f)$ to be
$$
  \ord_V(f) := \length_{\cO_{X,V}} \left( \frac{\cO_{X,V}}{\langle f \rangle}
\right).
$$
This determines a well-defined group homomorphism $\ord_V:
\Gamma(X,\mathcal{K}_X^{*}) \to \Z$. This length can be computed as follows. Choose $x \in V$ such that $x$ is smooth in $X$
and $(V,x)$ defines an irreducible germ. This germ is the zero set of an
irreducible $g \in \cO_{X,x}$. Then
$
  \ord_V(f) = \ord_{V,x}(f),
$
where $\ord_{V,x}(f)$ is the classical order of a meromorphic function at a
smooth point with respect to an irreducible subvariety of codimension one; it is
known to be given by the equality $f = g^{\ord} \cdot h \in \cO_{X,x}$ where $h
\nmid g$.

Now if $D$ is a Cartier divisor on $X$, one writes $\ord_V (D) = \ord_V(f_i)$
where $f_i$ is a local equation of $D$ on any open set $U_i$ with $U_i \cap V
\neq \emptyset$. This is well defined since $f_i$ is uniquely determined up to
multiplication by units and the order function is a homomorphism. Define the
{\em associated Weil divisor} of a Cartier divisor~$D$ by 
$$
\begin{array}{cccl}
T_X: & \CaDiv(X) & \longrightarrow & \WeDiv(X), \\[0.15cm]
& D & \mapsto & \displaystyle \sum_{V \subset X} \ord_V(D) \cdot [V],
\end{array}
$$
where the sum is taken over all codimension one irreducible subvarieties $V$ of
$X$; the mapping
$T_X$ is a homomorphism of abelian groups.

A Weil divisor is {\em principal} if it is the image of a principal Cartier
divisor under $T_X$; they form a subgroup of $\WeDiv(X)$. If $\Cl(X)$ denotes
the quotient group of their equivalence classes, then $T_X$ induces a morphism
$\Pic(X) \rightarrow \Cl(X)$.

These two homomorphisms ($T_X$ and the induced one) are in general neither
injective nor surjective.

\begin{ex}\label{ex_non-cartier_weil_divisor}
Let $X:=X(2;1,1)$ and consider the Weil divisor $D$ associated with $x=0$. Since $x$ does not define a function
on $X$, then $D$ is not a Cartier divisor. Since $x^2: X(2;1,1) \to \C$,
then $E:=\{(X(2;1,1),x^2)\}$ is a Cartier divisor
and it is easily seen that $T_X(E)=2 D$.
\end{ex}

Example~\ref{ex_non-cartier_weil_divisor} above illustrates the general behavior
of Cartier and Weil divisors on $V$-manifolds. The following theorem
allows us to identify both notions on $V$-manifolds after tensorizing by~$\mathbb{Q}$.

\begin{theo}[\cite{kjj-qcw}]\label{cartier_weil_XdA}
Let $X$ be a $V$-manifold. Then the notion of Cartier and Weil divisor coincide
over $\mathbb{Q}$. More precisely, the linear map 
$$
  T_X \otimes 1: \CaDiv(X) \otimes_{\Z} \mathbb{Q} \longrightarrow \WeDiv(X)
\otimes_{\Z} \mathbb{Q} 
$$
is an isomorphism of $\mathbb{Q}$-vector spaces. In particular, for a given Weil
divisor $D$ on $X$ there always exists $k \in \Z$ such that $kD \in \CaDiv(X)$.
\end{theo}

\begin{defi}
Let $X$ be a $V$-manifold. The vector space of $\mathbb{Q}$-Cartier divisors is
identified under $T_X$ with the vector space of $\mathbb{Q}$-Weil divisors. A
{\em $\mathbb{Q}$-divisor} on $X$ is an element in $\CaDiv(X) \otimes_{\Z}
\mathbb{Q} = \WeDiv(X) \otimes_{\Z} \mathbb{Q}.$ The set of all
$\mathbb{Q}$-divisors on $X$ is denoted by $\mathbb{Q}$-$\Div(X)$.
\end{defi}

In \cite{kjj-qcw}, we also give a way to construct the inverse of $T_X \otimes 1$.
\begin{nothing}\label{how_to_write_summarize}
Here we summarize how to write a Weil divisor as a $\mathbb{Q}$-Cartier divisor
where $X$ is an algebraic $V$-manifold.
\begin{enumerate} \setlength{\itemsep}{3pt}
\item Write $D = \sum_{i \in I} a_i [V_i] \in \WeDiv(X)$, where $a_i \in \Z$ and
$V_i \subset X$ irreducible. Also choose $\{U_j\}_{j\in J}$ an open covering of
$X$ such that $U_j = B_j/G_j$ where $B_j \subset \C^n$ is an open ball and $G_j$
is a {\bf small} finite subgroup of $GL(n,\C)$.

\item For each $(i,j) \in I \times J$ choose a {\bf reduced} polynomial $f_{i,j}: U_j
\to \C$
such that $V_i \cap U_j = \{ f_{i,j} = 0 \}$, then 
$$
[V_i|_{U_j}] = \frac{1}{|G_j|} \{ (U_j, f_{i,j}^{|G_j|}) \}.
$$

\item Identifying $\{(U_j, f_{i,j}^{|G_j|})\}$ with its image $\CaDiv(U_j)
\hookrightarrow \CaDiv(X),$ one finally writes $D$ as a sum of locally principal
Cartier divisors over $\mathbb{Q}$,
$$
  D = \sum_{(i,j) \in I \times J} \frac{a_i}{|G_j|} \{ (U_j, f_{i,j}^{|G_j|})
\}.
$$
\end{enumerate}
\end{nothing}

Let us apply this procedure to write the exceptional divisor of a
weighted blow-up (which is in general just a Weil divisor) 
as a $\mathbb{Q}$-Cartier divisor.

\begin{ex}\label{how_to_write_example}
Let $X$ be a surface with abelian quotient singularities. Let $\pi: \widehat{X}
\to X$ be the weighted blow-up at a point of type $(d;a,b)$ with respect to
$\w=(p,q)$. In general, the exceptional divisor $E:=\pi^{-1}(0)\cong \bP^1_{\w}
(d;a,b)$ is a Weil divisor on $\widehat{X}$ which does not correspond to a
Cartier divisor. Let us write $E$ as an element in $\CaDiv(\widehat{X})
\otimes_{\Z} \mathbb{Q}$.

As in~\ref{blow-up2-sing-ab}, assume $\pi := \pi_{(d;a,b),\w}:
\widehat{X(d;a,b)}_{\w} \to X(d;a,b)$. Assume also that $\gcd(p,q)=1$ and
$(d;a,b)$ is normalized. Using the
notation introduced in~\ref{blow-up2-sing-ab}, the space $\widehat{X}$ is covered by $\widehat{U}_1
\cup \widehat{U}_2$ and the first chart is given by
\begin{equation}\label{isom_1st_chart_dim2}
\begin{array}{rcl}
Q_1 := X \Big( \frac{pd}{e}; 1, \frac{-q+\beta pb}{e} \Big) & \longrightarrow &
\widehat{U}_1, \\[0.25cm] \,\big[ (x^e,y) \big] & \mapsto & \big[ ((x^p,x^q
y),[1:y]_{\w}) \big]_{(d;a,b)},
\end{array}
\end{equation}
where $e := \gcd(d,pb-qa)$, see~\ref{blow-up2-sing-ab} for details.

In the first chart, $E$ is the Weil divisor $\{x=0\} \subset Q_1$. Note that the
type representing the space $Q_1$ is in a normalized form and hence the
corresponding subgroup of $GL(2,\C)$ is small.

Following the discussion~\ref{how_to_write_summarize}, the divisor $\{x=0\}
\subset Q_1$ is written as an element in $\CaDiv(Q_1) \otimes_{\Z} \mathbb{Q}$
like $\frac{e}{pd} \{ (Q_1,x^{\frac{pd}{e}}) \}$, which is mapped to
$\frac{e}{pd} \{ (\widehat{U}_1, x^d) \} \in \CaDiv(\widehat{U}_1) \otimes_{\Z}
\mathbb{Q}$ under the isomorphism~(\ref{isom_1st_chart_dim2}).

Analogously $E$ in the second chart is $\frac{e}{qd} \{ ( \widehat{U}_2,y^d)\}$.
Finally one writes the exceptional divisor of $\pi$ as claimed,
$$
E  = \frac{e}{dp} \big\{ (\widehat{U}_1, x^d), (\widehat{U}_2, 1) \big\} + 
\frac{e}{dq} \big\{ (\widehat{U}_1, 1), (\widehat{U}_2, y^d) \big\} =
\frac{e}{dpq} \big\{ (\widehat{U}_1, x^{dq}), (\widehat{U}_2, y^{dp}) \big\}.
$$
\end{ex}

\begin{nothing}\textbf{Cartier divisors and sections of line bundles.}
Given a line bundle $\pi:E\to X$ and a non-zero meromorphic section $s:X\dashrightarrow E$, a 
Cartier divisor $\Div(s)$
can be constructed using the expression of $s$ in charts.
By its very construction, a line bundle~$\cO(D)$ is associated
with a Cartier divisor $D = \{ (U_i, f_i) \}_{i\in I}$; moreover one can also 
associate a class of meromorphic sections $S_D:=\{s_{D}:X\dasharrow\cO(D)\mid \Div(s_D)=D\}$;
once such a section is constructed any other one in $S_D$ is obtained by multiplying
by an element in $\Gamma(X,\cO_X^*)$. A divisor is effective if the sections in $S_D$ are holomorphic.
 \end{nothing}

Let $F:Y \to X$ be a morphism between two irreducible complex analytic
varieties. The {\em pull-back} of a Cartier divisor $D = \{(U_i,f_i)\}_{i\in I}$
on $X$ is defined by pulling back the local equations of $D$ as
$$
  F^{*} (D) = \big\{(F^{-1}(U_i), f_i \circ F|_{F^{-1}(U_i)} \big\}_{i \in I}
$$
and it is a Cartier divisor on $Y$ provided $F(Y) \nsubseteq \Supp(D)$.
If $F(Y) \subset \Supp(D)$ then we identify $F^*(Y)$ with $F^*\cO(D)$. If this line
bundle admits nonzero meromorphic sections we consider $F^*(Y)$ as a linear
equivalence class of Cartier Divisors.
Moreover, $F^{*}$ respects sums of divisors and preserves linear equivalence.
In the same way, pull-backs of $\mathbb{Q}$-Cartier divisors can be defined.
Note that $F^*\cO(D)=\cO(F^*D)$ and the same happens for sections.

\begin{remark}
Line bundles on projective varieties admit nonzero meromorphic sections.
It is also the case for a ball $B\subset\C^n$ since only the trivial bundle can be constructed.
\end{remark}

By the results in this section
the pull-back of a Weil divisor can be also constructed
for $V$-manifolds.

\section{Rational Intersection Number on $V$-Surfaces: Generalities}

Now we have all the necessary ingredients to develop a rational intersection theory on
varieties with quotient singularities. This section is devoted to working out all the details,
but the following illustrative example will be given.

\begin{ex}\label{example_X211_intersection}
Let $X:=X(2;1,1)$ and consider the Weil divisors $D_1: = \{x=0\}$ and $D_2: =
\{y=0\}$. Let us compute the Weil divisor associated with $j^{*}_{D_1} D_2$,
where $j_{D_1}: |D_1| \hookrightarrow X$ is the inclusion.
Following~\ref{how_to_write_summarize}, the divisor $D_2$ can be written as
$\frac{1}{2} \{(X,y^2)\}$. By definition, since $|D_1| \nsubseteq |D_2|$, the
pull-back is
$
  j^{*}_{D_1} D_2 = \frac{1}{2} \big\{(D_1, y^2|_{D_1}) \big\},
$
and its associated Weil divisor is 
$$
  T_{D_1} ( j^{*}_{D_1} D_2 )  = \frac{1}{2} \sum_{P \in D_1} \ord_P (
y^2|_{D_1} ) \cdot [P]  = \frac{1}{2} \ord_{[(0,0)]} ( y^2|_{D_1} ) \cdot
[(0,0)] \ =  \ \frac{1}{2} \cdot [(0,0)].
$$

Note that there is an isomorphism $D_1 = X(2;1) \simeq \C$, $[y] \mapsto y^2$,
and the function $y^2: D_1 \to \C$ is converted into the identity map $\C \to
\C$ under this isomorphism. Hence $\ord_{[(0,0)]} ( y^2|_{D_1} ) = 1$. It is
natural to define the (global and local) intersection multiplicity as
$
D_1 \cdot D_2 = (D_1 \cdot D_2)_{[(0,0)]} = \frac{1}{2}.
$
\end{ex}

\begin{defi}\label{definition_degree_Q-divisor}
Let $\mathcal{C}$ be an irreducible analytic curve. Given a Weil divisor on
$\mathcal{C}$ with finite support, $D: = \sum_{i=1}^r n_i \cdot [P_i]$, its {\em
degree} is defined as $\deg (D) = \sum_{i=1}^r n_i \in \Z$. The {\em degree
of a Cartier divisor} is the degree of its associated Weil divisor, that is, by
definition $\deg(D) := \deg ( T_{\mathcal{C}} D )$.

The degree map is a group homomorphism. If $\mathcal{C}$ is compact, the degree
of a principal divisor is zero and thus passes to the quotient yielding $\deg:
\Cl(\mathcal{C}) \to \Z$, cf.~\cite[Prop.~1.4]{Fulton98}.
\end{defi}

\begin{defi}
Let $X$ be an analytic surface and consider $D_1 \in \WeDiv(X)$ and $D_2 \in
\CaDiv(X)$. If $D_1$ is irreducible, then the {\em intersection number} is
defined as
$$
  D_1 \cdot D_2 := \deg ( j_{D_1}^{*} D_2 ) \in \mathbb{Z},
$$
where $j_{D_1} \hookrightarrow X$ denotes the inclusion and $j_{D_1}^{*}$ its
pull-back functor. The expression above extends linearly if $D_1$ is a finite
sum of irreducible divisors. This intersection number is only well defined if
$D_1 \nsubseteq D_2$ and $D_1 \cap D_2$ is finite, or if the divisor $D_1$ is
compact, cf.~\cite[Ch.~2]{Fulton98}.

In the case $D_1 \nsubseteq D_2$ the number $(D_1 \cdot D_2)_P:= \ord_P (
j_{D_1}^{*} D_2 )$ with $P \in D_1 \cap D_2$ is well defined too and it is called
{\em local intersection number} at $P$.
\end{defi}

\begin{defi}\label{definition_intersection_number}
Let $X$ be a $V$-manifold of dimension $2$ and consider $D_1, D_2 \in
\mathbb{Q}$-$\Div(X)$. The {\em intersection number} is defined as
$$
  D_1 \cdot D_2 := \frac{1}{k_1 k_2} ( k_1 D_1 \cdot k_2 D_2 ) \in \mathbb{Q},
$$
where $k_1, k_2 \in \mathbb{Z}$ are chosen so that $k_1 D_1 \in \WeDiv(X)$ and
$k_2 D_2 \in \CaDiv(X)$. Analogously, it is defined the {\em local intersection
number} at $P \in D_1 \cap D_2$, if the condition $D_1 \nsubseteq D_2$ is
satisfied. Idem the pull-back is defined by $F^{*} (D_1) := \frac{1}{k_1} (k_1
D_1)$ if $F: Y \to X$ is a proper morphism between two irreducible $V$-surfaces.
\end{defi}

\begin{remark}\label{bezoutglolo}
If $D_1 \nsubseteq D_2$ and $D_1 \cap D_2$ is finite, then the global and the
local intersection number at $P \in D_1 \cap D_2$ are defined, and indeed by
definition 
$$
  D_1 \cdot D_2 = \displaystyle \sum_{P \in D_1 \cap D_2} ( D_1 \cdot D_2 )_P.
$$
\end{remark}

In the following result the main usual properties of the intersection multiplicity
are collected. Their proofs are omitted since they are well known for the
classical case (i.e.~without tensoring with $\mathbb{Q}$), cf.~\cite{Fulton98},
and our generalization is based on extending the classical definition to
rational coefficients.

\begin{prop}\label{all_properties_inter_number}
Let $X$ be a $V$-manifold of dimension $2$ and $D_1, D_2, D_3 \in
\mathbb{Q}$-$\Div(X)$. Then the local and the global intersection numbers,
provided the indicated operations make sense according to
Definition{\rm~\ref{definition_intersection_number}}, satisfy the following
properties: ($\alpha \in \mathbb{Q}$, $P \in X$)
\begin{enumerate}[\rm(1)] 
\item\label{all_properties_inter_number1} The intersection product is \textbf{bilinear} over $\mathbb{Q}$.

\item\label{all_properties_inter_number2} {\bf Commutative:} If $D_1 \cdot D_2$ and $D_2 \cdot D_1$ are both
defined, then $D_1 \cdot D_2 = D_2 \cdot D_1$. Analogously $( D_1 \cdot D_2 )_P
= ( D_2 \cdot D_1 )_P$ if both local numbers are defined. 

\item\label{all_properties_inter_number3} {\bf Non-negative:} Assume $D_1$ and $D_2$ are effective, irreducible and
distinct. Then $D_1 \cdot D_2$ and $(D_1 \cdot D_2)_P$ are greater than or equal
to zero if they are defined. Moreover, $(D_1 \cdot D_2)_P = 0$ if and only if $P
\notin |D_1| \cap |D_2|$, and hence $D_1 \cdot D_2 = 0$ if and only if $|D_1|
\cap |D_2| = \emptyset$.

\item\label{all_properties_inter_number4} {\bf Non-rational:} If $D_2 \in \CaDiv(X)$ and $D_1 \in \WeDiv(X)$ then
$D_1 \cdot D_2$ and $(D_1 \cdot D_2)_P$ are integral numbers. By the commutative
property, the same holds if $D_1$ is a Cartier divisor and $D_2$ is a Weil
divisor. 

\item\label{all_properties_inter_number5} {\bf $\mathbb{Q}$-Linear equivalence:} Assume $D_1$ has compact support.
If $D_2$ and $D_3$ are $\mathbb{Q}$-linearly equivalent, i.e.~$[D_2] = [D_3] \in
\Pic(X) \otimes_{\Z} \mathbb{Q}$, then $D_1 \cdot D_2 = D_1 \cdot D_3$. Due to
the commutativity, the roles of $D_1$ and $D_2$ can be exchanged. In particular
$D_1 \cdot D_2 = 0$ for every principal $\mathbb{Q}$-divisor $D_2$.

\item\label{all_properties_inter_number6} {\bf Normalization:} Let $\nu: \widetilde{|D_1|} \to |D_1|$ be the
normalization of the support of $D_1$ and $j_{D_1} : |D_1| \hookrightarrow X$
the inclusion. Then $D_1 \cdot D_2 = \deg \big( j_{D_1} \circ \nu \big)^{*}
D_2$. Observe that in this situation the normalization is a smooth complex
analytic curve. 
\end{enumerate}
\end{prop}

\begin{remark}
This rational intersection multiplicity was first introduced by Mumford for normal
surfaces, see~\cite[Pag.~17]{Mumford61}. Our
Definition~\ref{definition_intersection_number} coincides with Mumford's because
it has good behavior with respect to the pull-back, see
Theorem~\ref{formula_pull-back} and a direct proof will be also given later, see Proposition \ref{prop_intmat}. The main advantage is that ours does not
involve a resolution  of the ambient space and, for instance, this allows us to
easily find formulas for the self-intersection numbers of the exceptional
divisors of weighted blow-ups, without computing any resolution, see
Proposition~\ref{formula_self-intersection}.
\end{remark}

\begin{theo}\label{formula_pull-back}
Let $F:Y \to X$ be a proper morphism between two irreducible $V$-manifolds of
dimension $2$, and $D_1, D_2 \in \mathbb{Q}$-$\Div(X)$. 
\begin{enumerate}[\rm(1)] \setlength{\itemsep}{3pt}
\item\label{formula_pull-back1} The cardinal of $F^{-1}(P)$, $P \in X$ being generic, is a finite
constant. This number is denoted by $\deg(F)$. 

\item\label{formula_pull-back2} If $D_1 \cdot D_2$ is defined, then so is the number $F^{*} ( D_1 ) \cdot
F^{*} ( D_2 )$. In such a case $F^{*} ( D_1 ) \cdot F^{*} ( D_2 ) = \deg (F) \,
( D \cdot E )$. 

\item\label{formula_pull-back3} If $(D_1 \cdot D_2)_P$ is defined for some $P \in X$, then so is $( F^{*}
(D_1) \cdot F^{*}(D_2) )_Q$, $\forall Q \in F^{-1}(P)$, and 
$\sum_{Q \in F^{-1}(P)} ( F^{*}(D_1) \cdot F^{*}(D_2) )_Q = \deg(F) (D_1 \cdot
D_2)_P$.
\end{enumerate}
\end{theo}

The rest of this section is devoted to reviewing some classical results concerning
the intersection multiplicity, namely, the computation of the local intersection
number at a smooth point, the self-intersection numbers of the exceptional
divisors of blow-ups at a smooth point, and the classical B{\'e}zout's Theorem on
$\bP^2$. Afterwards, these results are generalized in the upcoming sections.


\begin{nothing}\label{local_inter_number_smooth} \textbf{Local intersection number at a
smooth point.} Let $X$ be a smooth analytic surface. Consider $D_1$, $D_2$ two
effective (Cartier or Weil)\footnote{Recall that on smooth analytic varieties,
Cartier and Weil divisors are identified and their equivalence classes coincide
under this identification, i.e.~$\Pic(X) = \Cl(X)$.} divisors on $X$ and $P \in X$ a point.
The divisor $D_i$ is locally given by a holomorphic function $f_i$, $i=1,2$, in
a neighborhood of $P$. Then $(D_1 \cdot D_2)_P$ equals
$$
  \ord_{P} (f_2|D_1) = \length_{\cO_{D_1,P}} \left( \frac{\cO_{D_1,P}}{f_2|_{D_1}}
\right) = \dim_{\C} \left( \frac{\cO_{X,P}}{\langle f_1, f_2 \rangle} \right).
$$

Moreover, $X$ being a smooth variety, $\cO_{X,P}$ is isomorphic to $\C \{ x,y \}$
and hence the previous dimension can be computed, for instance, by means of
Gröbner bases with respect to local orderings.
\end{nothing}

\begin{nothing}\label{classic_blow-ups}\textbf{Classical blow-up at a smooth
point.} Let $X$ be a smooth analytic surface. Let $\pi: \widehat{X} \to X$ be
the classical blow-up at a (smooth) point $P$. Consider $C$ and $D$ two (Cartier
or Weil) divisors on $X$ with multiplicities $m_C$ and $m_D$ at $P$.  Denote by
$E$ the exceptional divisor of $\pi$, and by $\widehat{C}$ (resp.~$\widehat{D}$)
the strict transform of $C$ (resp.~$D$). Then, 
\begin{enumerate} \setlength{\itemsep}{3pt}
\item $E\cdot \pi^{*}(C) = 0$, \qquad $\pi^{*}(C) = \widehat{C} + m_C E$, \qquad
$E \cdot \widehat{C} = m_C$.
\item $E^2 = -1$, \qquad $\widehat{C} \cdot \widehat{D} = C\cdot D - m_C m_D$,
\qquad $\widehat{D}^2 = D^2 - m_D^2$,\, ($D$ compact).
\end{enumerate}
The first properties follow from the
local equations of the blow-up, since $C$ is principal near $P$. The second ones are easy consequences of the first ones.
\end{nothing}

\begin{nothing}\label{classic_bezout}\textbf{B{\'e}zout's Theorem on $\bP^2$.}
Every
analytic (Cartier or Weil) divisor on $\bP^2$ is algebraic and thus can be
written as a difference of two effective divisors. On the other hand, every
effective divisor is defined by a homogeneous polynomial. The {\em degree of an
effective divisor on $\bP^2$} is the degree $\deg(F)$ of the corresponding
homogeneous polynomial. This degree map is extended linearly yielding a group
homomorphism $\deg: \Div(\bP^2) \to \Z$ that characterizes the linear equivalence
classes in the following sense: $\forall D_1, D_2 \in \Div(\bP^2)$,
\begin{equation}\label{characterization_divisor_P2}
\Big[ \, [D_1] = [D_2] \in \Pic(\bP^2) = \Cl(\bP^2) \, \Longleftrightarrow \,
\deg(D_1) = \deg(D_2) \, \Big].
\end{equation}

Let $D_1$, $D_2$ be two divisors on $\bP^2$, then $D_1 \cdot D_2 = \deg(D_1)
\deg(D_2)$. In particular, the self-intersection number of a divisor $D$ on
$\bP^2$ is given by $D^2 = \deg(D)^2$. In addition, if $|D_1| \nsubseteq |D_2|$,
then $|D_1| \cap |D_2|$ is a finite set of points and, by Remark \ref{bezoutglolo}, one has
$$
  \deg(D_1) \deg(D_2) = D_1 \cdot D_2 = \sum_{P \in |D_1| \cap |D_2|} ( D_1
\cdot D_2)_P.
$$

The proof of this result is an easy consequence
of~\ref{all_properties_inter_number}, and the fact that~$D_i$ is linearly
equivalent to $\deg(D_i) L_i$, where $L_i$ is a linear form, $i=1,2$,
by~(\ref{characterization_divisor_P2}). 
\end{nothing}

In what follows, we generalize the classical results
of~\ref{local_inter_number_smooth}, \ref{classic_blow-ups}
and~\ref{classic_bezout} to $V$-manifolds, weighted blow-ups and quotient
weighted projective planes, respectively.

We start this generalization providing the computation
of local intersection numbers for quotient surfaces. 
Let $X$ be an algebraic $V$-manifold of dimension $2$. Consider $D_1$ and $D_2$
two effective $\mathbb{Q}$-divisors on $X$, and $P \in X$ a point. The divisor
$D_i$ is locally given in a neighborhood of $P$ by a reduced polynomial $f_i$,
$i=1,2$. On the other hand the
point $P$ can be assumed to be a normalized type of the form $(d;a,b)$. Hence
the computation of $(D_1 \cdot D_2)_P$ is reduced to the following particular
case.

\begin{nothing}\label{computation_local_number_V-surface}
\textbf{Local intersection number on $X(d;a,b)$.} 
Denote by $X$ the cyclic quotient space $X(d;a,b)$ and
consider two divisors $D_1 = \{ f_1 = 0 \}$ and $D_2 = \{ f_2 = 0 \}$ given by
$f_1,f_2\in\C\{x,y\}$ reduced.
Assume that, $(d; a, b)$ is normalized, $D_1$ is
irreducible, $f_1$ induces a function on $X$, and 
finally that $D_1 \nsubseteq D_2$.


Then as Cartier divisors $D_1 = \{(X,f_1) \}$, $D_2 = \frac{1}{d} \{(X,f_2^d)
\}$, and the pull-back is $j^{*}_{D_1} D_2 = \frac{1}{d} \{(D_1,f_2^d|_{D_1})\}$.
Following the definition, the local number $(D_1 \cdot D_2)_{[P]}$ equals
$$
  \frac{1}{d} \ord_{[P]} ( f_2^d|_{D_1} ) = \frac{1}{d} \length_{\cO_{D_1,[P]}}
\left( \frac{\cO_{D_1,[P]}}{\langle f_2^d|_{D_1} \rangle} \right) = \frac{1}{d}
\dim_{\C} \left( \frac{\cO_{X,[P]}}{\langle f_1, f_2^d \rangle} \right).
$$

There is an isomorphism of local rings if $P = (\alpha, \beta ) \neq (0,0)$,
\begin{eqnarray*}
  \cO_{X,[P]} & \longrightarrow & \cO_{\C^2,(\alpha,\beta)}, \\
  (x,y) & \mapsto & (x+\alpha, y+\beta),
\end{eqnarray*}
and for $P=(0,0)$ one has $\cO_{X,[(0,0)]} \cong \C\{ x,y \}^{\mu_d}$.

Also $\frac{1}{d} \dim_{\C} ( \C\{ x, y \}/ \langle f_1, f_2^d \rangle )$
coincides with $\dim_\C \C\{x,y\}/\langle f_1, f_2 \rangle$. 
So finally,
$$
(D_1 \cdot D_2)_{[P]} =
\begin{cases}
\displaystyle \frac{1}{d} \dim_{\C} \left( \frac{\C\{x,y\}^{\mu_d}}{\langle f_1,
f_2^d \rangle} \right), & P = (0,0) \, ; \\[0.5cm]
\displaystyle \dim_{\C} \left( \frac{\C\{x-\alpha,y-\beta\}}{\langle f_1, f_2
\rangle} \right), & P = (\alpha, \beta) \neq (0,0) \, .
\end{cases}
$$

Analogously, if $f_1$ does not define a function on $X$, for computing the
intersection number at $[(0,0)]$ one substitutes $f_1$ by $f_1^d$ and divides
the result by $d$.
\end{nothing}

Another way to calculate $(D_1 \cdot D_2)_{[(0,0)]}$ is to consider the natural
projection $\pr: \C^2 \to X(d;a,b)$ and apply the local pull-back formula, see
Theorem~\ref{formula_pull-back}\eqref{formula_pull-back3}. Indeed, let $\widetilde{D}_i$ be the
pull-back divisor of $D_i$ under the projection, $i=1,2$. Then,
$$
  (D_1 \cdot D_2)_{[(0,0)]} = \frac{1}{d} ( \widetilde{D}_1 \cdot
\widetilde{D}_2 )_{(0,0)} = \frac{1}{d} \dim_{\C} \left(
\frac{\C\{x,y\}}{\langle f_1, f_2 \rangle} \right).
$$
In particular, combining the two expressions obtained for $(D_1 \cdot
D_2)_{[(0,0)]}$, if two polynomials $f$ and $g$ define functions on $X$, then
$$
  \dim_{\C} \left( \frac{\C\{x,y\}^{\mu_d}}{\langle f, g \rangle} \right) =
  \frac{1}{d} \dim_{\C} \left( \frac{\C\{x,y\}}{\langle f, g \rangle} \right).
$$

As in the smooth case, all the preceding dimensions can be computed by means of
Gröbner bases with respect to local orderings.

\begin{ex}
Let $X = X(2;1,1)$ and consider the Weil divisors $D_1 = \{x=0\}$ and $D_2 =
\{y=0\}$. In Example~\ref{example_X211_intersection} it is showed, by directly using
the definition of the intersection product, that $(D_1 \cdot
D_2)_{[(0,0)]} = \frac{1}{2}.$

Two expressions have been obtained for computing this local number:
\begin{itemize}\setlength{\itemsep}{7.5pt}
\item $\displaystyle (D_1 \cdot D_2)_{[(0,0)]} = \frac{1}{2} \dim_{\C} \left(
\frac{\C\{x,y\}}{\langle x, y \rangle} \right) = \frac{1}{2}$.
\item $\displaystyle (D_1 \cdot D_2)_{[(0,0)]} = \frac{1}{4} \dim_{\C} \left(
\frac{\C\{ x, y \}^{\mu_2}}{\langle x^2, y^2 \rangle} \right) = \frac{1}{4}
\cdot 2 = \frac{1}{2}$.
\end{itemize}
For the second equality note that $\C \{ x,y \}^{\mu_2} = \C \{ x^2, y^2, xy \}$.
\end{ex}


\section{Intersection Numbers and Weighted Blow-ups}
                                                                                                                                                                                                                                  
Previously weighted blow-ups were introduced as a tool for computing embedded
$\Q$-resolutions. To obtain information about the corresponding embedded
singularity, an intersection theory on $V$-manifolds has been developed. Here we
calculate self-intersection numbers of exceptional divisors of weighted blow-ups
on analytic varieties with abelian quotient singularities, see Proposition
\ref{formula_self-intersection}.

We state some preliminary lemmas separately so that the proof of the main result
of this section becomes simpler.

\begin{lemma}\label{lem-except-pb}
Let $X$ be an analytic surface with abelian quotient singularities and let $\pi:
\widehat{X} \to X$ be  a weighted blow-up at a point $P \in X$. Let $C$ be a
$\mathbb{Q}$-divisors on $X$ and $E$ 
the exceptional divisor of $\pi$. Then, $E \cdot \pi^{*}(C) = 0$.  
\end{lemma}
\begin{proof}
Using Proposition~\ref{all_properties_inter_number}\eqref{all_properties_inter_number5} 
it can be proven as in the smooth case since $\pi^{*} (C)$ is locally principal as
$\mathbb{Q}$-divisor on~$\widehat{X}$. 
\end{proof}

\begin{lemma}\label{lemma_formula_self-intersection}
Let $h: Y \to X$ be a proper morphism between two irreducible $V$-manifolds of
dimension $2$.

Consider $\pi_X: \widehat{X} \to X$ (resp.~$\pi_Y: \widehat{Y} \to Y$) a
weighted blow-up at a point of $X$ (resp.~$Y$) and take $C_X$ a
$\mathbb{Q}$-divisor on $X$. Denote by $E_X$ (resp.~$E_Y$) the exceptional
divisor of $\pi_X$ (resp.~$\pi_Y$), and $\widehat{C}_X$ the strict transform of
$C_X$.

Let us suppose that there exist two rational numbers, $e$ and $\nu$, and a
finite proper morphism $H: \widehat{Y} \to \widehat{X}$ completing the
commutative diagram

\begin{minipage}[b]{0.45\linewidth}
$$
\begin{array}{cc}
\xymatrix{
\ar @{} [dr] | {\#}
\widehat{Y} \ar[r]^{H} \ar[d]_{\pi_Y} & \widehat{X} \ar[d]^{\pi_X} \\
Y \ar[r]_{h} & X
}
\end{array}
$$
such that: 
\begin{enumerate}[\rm(a)]
\item\label{lemma_formula_self-intersection-a}  $H^{*} ( E_X ) = e E_Y$,
\item\label{lemma_formula_self-intersection-b}  $\pi_Y^{*} (
h^{*} ( C_X ) ) = H^{*}(\widehat{C}_X) + \nu E_Y$.
\end{enumerate}
\end{minipage}
\hspace{0.5cm}
\begin{minipage}[b]{0.45\linewidth}
Then the following
equalities hold:
\begin{enumerate}[\rm(1)] \setlength{\itemsep}{3pt}
\item\label{lemma_formula_self-intersection1} $\pi_X^{*}(C_X) = \widehat{C}_X + \frac{\nu}{e} E_X$, 
\item\label{lemma_formula_self-intersection2} $E_X \cdot \widehat{C}_X = \frac{- e \, \nu \ }{\deg (h)} E_Y^2$, 
\item\label{lemma_formula_self-intersection3} $E_X^2 = \frac{e^2}{\deg (h)} E_Y^2$.
\end{enumerate}
\end{minipage}
\end{lemma}

\begin{proof}
For \eqref{lemma_formula_self-intersection1} 
note the total transform $\pi_X^{*}(C_X)$ can always be written as
$\widehat{C}_X + m E_X$ for some $m \in \mathbb{Q}$. Considering its pull-back
under $H^{*}$ one obtains two expressions for the same $\mathbb{Q}$-divisor on
$\widehat{Y}$, 
$$
\begin{array}{l}
H^{*} ( \pi_X^{*} ( C_X ) ) \stackrel{\text{diagram}}{=}  \pi_Y^{*} ( h^{*} (
C_X ) ) \stackrel{\text{(b)}}{=} H^{*} (\widehat{C}_X) + \nu E_Y, \\[0.1cm]
H^{*} ( \widehat{C}_X + m E_X ) = H^{*}(\widehat{C}_X) + m H^{*} (E_X)
\stackrel{\text{(a)}}{=} H^{*} (\widehat{C}_X) + m e E_Y.
\end{array}
$$
It follows that $m = \frac{\nu}{e}$.

For \eqref{lemma_formula_self-intersection2} first note that $\deg (H) = \deg (h)$.
>From Lemma~\ref{lem-except-pb}, one has that $E_Y \cdot
\pi_Y^{*} ( h^{*} (C_X) ) = 0$. On the other hand, $H$ being proper,
Theorem~\ref{formula_pull-back}\eqref{formula_pull-back2} can be applied thus obtaining
$$
  \deg ( h )  ( E_X \cdot \widehat{C}_X ) = H^{*} (E_X) \cdot H^{*}
(\widehat{C}_X) \stackrel{\text{(a)-(b)}}{=} e E_Y \cdot \big[ \pi_Y^{*} ( h^{*}
(C_X) ) - \nu E_Y \big] = - e \nu E_Y^2.
$$
Analogously $\deg (h) E_X^2 = H^{*} (E_X)^2 = e^2 E_Y^2$ and 
\eqref{lemma_formula_self-intersection3} follows.
\end{proof}

Now we are ready to present the main result of this section.

\begin{prop}\label{formula_self-intersection}
Let $X$ be an analytic surface with abelian quotient singularities and let $\pi:
\widehat{X} \to X$ be the $(p,q)$-weighted blow-up at a point $P \in X$ of type
$(d;a,b)$. Assume $\gcd(p,q) = 1$ and $(d;a,b)$ is a normalized type,
i.e.~$\gcd(d,a) = \gcd(d,b)=1$. Also write $e = \gcd(d, pb-qa)$.

Consider two $\mathbb{Q}$-divisors $C$ and $D$ on $X$. As usual, denote by $E$
the exceptional divisor of $\pi$, and by $\widehat{C}$ (resp.~$\widehat{D}$) the
strict transform of $C$ (resp.~$D$). Let $\nu$ and $\mu$ be the
$(p,q)$-multiplicities of $C$ and $D$ at $P$, i.e.~$x$ (resp.~$y$) has
$(p,q)$-multiplicity $p$ (resp.~$q$). Then there are the following equalities:
\vspace{0.2cm}

\begin{minipage}[t]{0.45\linewidth}
\begin{enumerate}[\rm(1)]
\item\label{formula_self-intersection2} $\displaystyle \pi^{*}(C) = \widehat{C} + \frac{\nu}{e} E$. 
\item\label{formula_self-intersection3} $\displaystyle E \cdot \widehat{C} = \frac{e \nu}{d p q}$. 
\end{enumerate}
\end{minipage}
\hspace{0.5cm}
\begin{minipage}[t]{0.45\linewidth}
\begin{enumerate}[\rm(1)]
\setcounter{enumi}{2}
\item\label{formula_self-intersection4} $\displaystyle E^2 = - \, \frac{e^2}{dpq}$. 
\item\label{formula_self-intersection5} $\displaystyle \widehat{C} \cdot \widehat{D} = 
C \cdot D - \frac{\nu \mu}{dpq}$.
\end{enumerate}
\end{minipage}
In addition, if $D$ has compact support then $\displaystyle \widehat{D}^2 = D^2 - \frac{\mu^2}{dpq}$.
\end{prop}

\begin{proof}
The item~\eqref{formula_self-intersection5}, and final conclusion,
are an easy consequence of \eqref{formula_self-intersection2}-\eqref{formula_self-intersection4} 
and the fact that $\pi^{*} (C) \cdot \pi^{*} (D) = C \cdot D$.

For the rest of the proof, one assumes that $\pi := \pi_{X}:
\widehat{X(d;a,b)}_{\w} \longrightarrow X(d;a,b)$ is the weighted blow-up at
the origin of $X(d;a,b)$ with respect to $\w = (p,q)$. Now the idea is to apply
Lemma~\ref{lemma_formula_self-intersection} to the commutative diagram
$$
\begin{array}{cc}
\xymatrix{
\ar @{} [dr] | {\#}
\hspace{-0.75cm} \widehat{Y} := \widehat{\C}^2 \ar[r]^<<<<<{H} \ar[d]_{\pi_Y} &
\widehat{X(d;a,b)}_{\w}  =: \widehat{X}  \hspace{-1.75cm} \ar[d]^{\pi_X} \\
\hspace{-0.75cm} Y := \C^2 \ar[r]_<<<<<{h} & X(d;a,b) =: X \hspace{-1.75cm}
}
\end{array}
$$
where $H$ and $h$ are the morphisms defined by 
$$
\begin{array}{ccc}
((x,y),[u:v]) & \stackrel{H}{\longmapsto} &
[((x^{p},y^{q}),[u^{p}:v^{q}])_{\w}]_{(d;a,b)}; \\[0.2cm]
(x,y) & \stackrel{h}{\longmapsto} & [(x^{p},y^{q})]_{(d;a,b)},
\end{array} 
$$
and $\pi_Y$ is the classical blowing-up at the origin. In this situation
$E_Y^2=-1$. The claim is reduced to the calculation of $\deg(h)$ and the
verification of the conditions 
\eqref{lemma_formula_self-intersection-a}-\eqref{lemma_formula_self-intersection-b} 
of Lemma~\ref{lemma_formula_self-intersection}.

The degree is $\deg(h) = pq \cdot \deg \left[ \pr: \C^2 \to X(d;a,b) \right] =
dpq$. For~\eqref{lemma_formula_self-intersection-a}, first recall the decompositions
\begin{equation}\label{decomp_blowup_Xdabw}
\widehat{X(d;a,b)}_{\w} = \widehat{U}_1 \cup \widehat{U}_2, \qquad
\widehat{\C}^2 = U_1 \cup U_2.
\end{equation}
By Example~\ref{how_to_write_example}, one writes the exceptional divisor of
$\pi_X$ as 
$$
E_X = \frac{e}{dpq} \left\{ (\widehat{U}_1, x^{dq}), (\widehat{U}_2, y^{dp})
\right\}.
$$
Hence its pull-back under $H$, computed by pulling back the local equations, is
$$
H^{*} (E_X) = \frac{e}{dpq}\, \Big\{ (U_1, x^{dpq}), (U_2, y^{dpq}) \Big\} = e\,
\Big\{ (U_1,x), (U_2,y) \Big\} = e E_Y.
$$
Finally for \eqref{lemma_formula_self-intersection-b} one uses local equations to check $\pi^{*}_Y ( h^{*} (C)) =
H^{*} ( \widehat{C} ) + \nu E_Y$. Suppose the divisor $C$ is locally given by a
meromorphic function $f(x,y)$ defined on a neighborhood of the origin of
$X(d;a,b)$; note that $\nu = \ord_{(p,q)} (f)$. The charts associated with the
decompositions~(\ref{decomp_blowup_Xdabw}) are described in detail in
\ref{blow-up2-sing-ab}. As a summary we recall here the first chart of
each blowing-up:
$$
\begin{array}{c|rcl}
\pi_X & Q_1 := X \left( \begin{array}{c|cc} p & -1 & q \\ pd & a & pb-qa
\end{array} \right) & \longrightarrow & \widehat{U}_1, \\[0.5cm] & \big[ (x,y)
\big] & \mapsto & \big[ ((x^p, x^q y), [1:y]_{\w}) \big]. \\
\multicolumn{4}{c}{} \\
\pi_Y & \C^2 & \longrightarrow & U_1, \\[0.15cm] & (x,y) & \mapsto &
((x,xy),[1:y]).
\end{array}
$$
Note that $H$ respects the decompositions and takes the form $(x,y) \mapsto
[(x,y^q)]$ in the first chart. Then one has the following local equations for
the divisors involved:
$$
\renewcommand{\arraystretch}{1.25}
\begin{array}{c|c|c} 
\text{\bf Divisor} & \text{\bf Equation} & \text{\bf Ambient space} \\ \hline
h^{*} (C) & f(x^p,y^q) = 0 & \C^2 \\[0.1cm] 
\pi_Y^{*} (h^{*}(C)) & f(x^p, x^q y^q) = 0 & \C^2 \cong U_1 \\[0.1cm] 
\widehat{C} & \displaystyle \frac{f(x^p, x^q y)}{x^\nu} = 0 & Q_1 \cong \widehat{U}_1 \\[0.2cm] 
H^{*}( \widehat{C} ) & \displaystyle \frac{f(x^p,x^q y^q)}{x^\nu} = 0 & \C^2 \cong U_1 \\[0.2cm] 
E_Y & x=0 & \C^2 \cong U_1 \\ 
\end{array}
$$

>From these local equations \eqref{lemma_formula_self-intersection-b} 
is satisfied and now the proof is complete.
\end{proof}

\begin{nothing} Let us discuss two special cases of
Proposition~\ref{formula_self-intersection} when $P \in X$ is smooth and the point $P$ is of type $(d;p,q)$ with
$\gcd(d,p) = \gcd(d,q) = 1$. Consider the weighted blow-up 
$\pi := \pi_{\w}: \widehat{\C}^2_{\w} \longrightarrow \C^2$
(resp.  $\pi: = \pi_{\w,d}: \widehat{\C}^2_{\w,d} \to X(d;p,q)$).
The following properties hold:
\begin{enumerate} 
\item $E \cdot \pi^{*}(C) = 0$ (in both cases).
\item $\pi^{*}(C) = \widehat{C} + \nu E$ (resp. $\pi^{*} (C) = \widehat{C} + \frac{\nu}{d} E$).
\item $E \cdot \widehat{C} = \frac{\nu}{p q}$ (in both cases).
\item $E^2 = - \frac{1}{pq}$ (resp. $E^2 = - \frac{d}{pq}$).
\item $\widehat{C} \cdot \widehat{D} = C \cdot D - \frac{\nu \mu}{pq}$ 
(resp. $\widehat{C} \cdot \widehat{D} = C \cdot D - \frac{\nu \mu}{dpq}$).
\end{enumerate}

%
%

\end{nothing}

\begin{ex}
We compute now the self-intersection of the divisors in Examples~\ref{example_2} and~\ref{example_3}. 
After the first blow-up (of type $(q,p)$ over a smooth point) the divisor $\E_1$ in Example~\ref{example_2} 
has self-intersection~$\frac{-1}{p q}$.
Let us consider the second blow-up, of type $(p,q^2-p^2)$ over a point
of type $(q;p,q^2-p^2)$; the exceptional divisor is $\E_2$ and its self-intersection is $-\frac{q}{p(q^2-p^2)}$.
The strict transform of $\E_1$ has multiplicity $p$ and hence its self-intersection is
$\frac{-1}{p q}-\frac{p}{q(q^2-p^2)}=-\frac{q}{p(q^2-p^2)}$, as it should be from the symmetry of the equation.

Let us consider now Example~\ref{example_3}. The first blow-up is the same as above. The second one 
is of type $(s,r q -p s)$ over a point of type $(q;-1,p)$. The self-intersection of $\E_2$ is
$-\frac{q}{s(r q -p s)}$. The strict transform of $\E_1$ has multiplicity $s$ and hence its self-intersection is
$\frac{-1}{p q}-\frac{s}{q(r q-p s)}=-\frac{r}{p(r q-p s)}$.
\end{ex}

\begin{ex}
Let us consider the following divisors on $\C^2$,
$$
C_1  = \{ ((x^3-y^2)^2-x^4y^3) = 0 \}, \quad		
C_2  = \{ x^3-y^2= 0 \}, 
$$
$$
C_3  = \{ x^3+y^2 = 0 \}, \quad
C_4  = \{ x = 0 \}, \quad
C_5  = \{ y=0 \}.
$$

We shall see that the local intersection numbers $(C_i \cdot C_j)_{0}$, $i,j \in
\{ 1,\ldots,5 \}$, $i\neq j$, are encoded in the intersection matrix associated
with any embedded $\Q$-resolution of $C = \bigcup_{i=1}^5 C_i$.

Let $\pi_1: \C^2_{(2,3)} \to \C^2$ be the $(2,3)$-weighted blow-up at the
origin. The new space has two cyclic quotient singular points of type $(2;1,1)$
and $(3;1,1)$ located at the exceptional divisor~$\mathcal{E}_1$. The local
equation of the total transform in the first chart is given by the function
$$
x^{29} \ ((1-y^2)^2-x^5 y^3) \ (1-y^2) \ (1+y^2) \ y: \, X(2;1,1)
\longrightarrow \C, 
$$
where $x=0$ is the equation of the exceptional divisor and the other factors
correspond in the same order to the strict transform of $C_1$, $C_2$, $C_3$,
$C_5$ (denoted again by the same symbol). To study the strict transform of $C_4$
one needs the second chart, the details are left to the reader.

Hence $\mathcal{E}_1$ has multiplicity $29$ and self-intersection number
$-\frac{1}{6}$; the divisor intersects transversally $C_3$, $C_4$ and $C_5$ at three
different points, while it intersects $C_1$ and $C_2$ at the same smooth
point~$P$, different from the other three. The local equation of the divisor
$\mathcal{E}_1 \cup C_2 \cup C_1$ at this point $P$ is $\, x^{29} \, y \,
(x^5-y^2) = 0$, see Figure~\ref{fig1} below.

\begin{figure}[ht]
\centering
\includegraphics{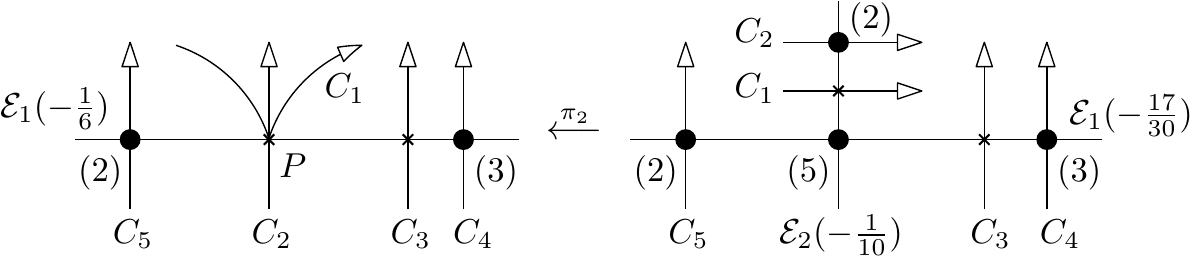}
\caption{Embedded $\Q$-resolution of $C = \bigcup_{i=1}^5 C_i \subset
\C^2$.}\label{fig1}
\end{figure}

Let $\pi_2$ be the $(2,5)$-weighted blow-up at the point $P$ above. The new
ambient space has two singular points of type $(2;1,1)$ and $(5;1,2)$. The local
equations of the total transform of $\mathcal{E}_1 \cup C_2 \cup C_1$ are given
by the functions in Table~\ref{tbl-function}.

\begin{table}[ht]
\begin{center}\renewcommand{\arraystretch}{1.5} 
\begin{tabular}{|c|} \hline
1st chart \\ \hline 
$\underbrace{x_{}^{73}}_{\mathcal{E}_2} \cdot \underbrace{y}_{C_2} \cdot \
\underbrace{(1-y^2)}_{C_1} \, : \, X(2;1,1) \longrightarrow \C$ \\ \hline
\end{tabular}
\begin{tabular}{|c|} \hline
2nd chart \\ \hline 
$\underbrace{x_{}^{29}}_{\mathcal{E}_1} \cdot
\underbrace{y^{73}}_{\mathcal{E}_2} \cdot \ \underbrace{(x^5_{}-1)}_{C_1} \, :
\, X(2;1,1) \longrightarrow \C$ \\ \hline
\end{tabular}
\end{center}
\caption{Equations of the total transform}
\label{tbl-function}
\end{table}
Thus the new exceptional divisor $\mathcal{E}_2$ has multiplicity $73$ and it
intersects transversally the strict transform of $C_1$, $C_2$ and
$\mathcal{E}_1$. 
Hence the composition $\pi_2 \circ \pi_1$ is an embedded $\Q$-resolution of $C =
\bigcup_{i=1}^5 C_i \subset \C^2$. Figure~\ref{fig1} above illustrates the whole
process.

As for the self-intersection numbers, $\mathcal{E}_2^2 = - \frac{1}{10}$ and
$\mathcal{E}_1^2 = - \frac{1}{6} - \frac{2^2}{1\cdot 2 \cdot 5} = -
\frac{17}{30}$.
The intersection matrix associated with the embedded $\Q$-resolution obtained
and its opposite inverse are
 $$
 A = \left( \begin{matrix} - 17/30 & 1/5 \\ 1/5 & - 1/10 \end{matrix} \right),
\qquad B = -A^{-1} = \left( \begin{matrix} 6 & 12 \\ 12 & 34 \end{matrix}
\right).
 $$

Now one observes that the intersection number is encoded in $B$ as follows. For
$i = 1,\ldots, 5$, 
set $k_i\in\{1,\ldots,5\}$ such that $\emptyset \neq C_i \cap \mathcal{E}_{k_i}
=: \{ P_i \}$. Denote by $d(C_i)$ the index of $P_i$, see Definition \ref{defiindex}.
Then,
$$
\left( C_i \cdot C_j \right)_0 = \frac{b_{k_i, k_j}}{d(C_i) \, d(C_j)}.
$$
One has $(k_1, \ldots, k_5) = (2,2,1,1,1)$ and $( d(C_1), \ldots, d(C_5)) =
(1,2,1,3,2)$.
Hence, for instance, $$( C_1 \cdot C_2 )_0  = \frac{b_{k_1, k_2}}{d(C_1) \,
d(C_2)} = \frac{b_{22}}{1 \cdot 2} = \frac{34}{2} = 17,$$ which is indeed the
intersection multiplicity at the origin of $C_1$ and $C_2$. Analogously for the
other indices.
\end{ex}


\begin{remark}
Consider the group action of type $(5;2,3)$ on $\C^2$. The previous plane curve
$C$ is invariant under this action and then it makes sense to compute an
embedded $\Q$-resolution of $\overline{C}:=C/\mu_5 \subset X(5;2,3)$. Similar calculations as in the previous example, lead to a figure as  the one obtained above with the following relevant differences: 
\begin{itemize}
\item $\mathcal{E}_1 \cap \mathcal{E}_2$ is a smooth point. \vspace{0.1cm}
\item $\mathcal{E}_1$ (resp.~$\mathcal{E}_2$) has self-intersection number
$-\frac{17}{6}$ (resp.~$-\frac{1}{2}$).
\item The intersection matrix is $A' = \left(\begin{smallmatrix} -17/6 & 1 \\ 1
& -1/2 \end{smallmatrix}\right)$ and its opposite inverse is $B' = -(A')^{-1} =
\left(\begin{smallmatrix} 6/5 & 12/5 \\ 12/5 & 34/5 \end{smallmatrix}\right)$.
\end{itemize}
Hence, for instance, $(\overline{C}_1 \cdot \overline{C}_2)_0 = \frac{b'_{22}}{1
\cdot 2} = \frac{34/5}{2} = \frac{17}{5}$, which is exactly the intersection
number of the two curves, since that local number can be also computed as
$(\overline{C}_1 \cdot \overline{C}_2)_0 = \frac{1}{5} (C_1 \cdot C_2)_0$.  
\end{remark}

The previous results correspond to  Mumford's definition~\cite{Mumford61}.
Let us fix $X:=X(d;a,b)$ and let us consider $\pi:\hat{X}\to X$ a sequence
of weighted blow-ups. Let $\E_1,\dots,\E_r$ be the set of exceptional components
and let $A:=\left(\E_i\cdot\E_j\right)_{1\leq i,j\leq r}$ be the intersection matrix in $\hat{X}$;
it is a negative definite matrix with rational coefficients. We may restrict $X$ to a small neighborhood of the origin.
An $\hat{X}$-curvette $\gamma_i$ of $\E_i$ is a Weil divisor obtained by considering a disk transversal to 
a  point of $\E_i\setminus\bigcup_{j\neq i}\E_j$ and $\delta_i=\pi(\gamma_i)$ is called an $X$-curvette of $\E_i$;
the index $d(\gamma_i):=d(\delta_i$) is the order of the cyclic group associated with $\gamma_i\cap\E_i$.
We say that $(\gamma_i,\gamma'_j)$ form a pair of $\hat{X}$-curvettes for $(\E_i,\E_j)$ if
they are disjoint curvettes for each divisor; in that case their images in $X$ form
a pair $(\delta_i,\delta'_j)$ $X$-curvettes.

\begin{prop}\label{prop_intmat}
Let $B:=-A^{-1}=\left(b_{i j}\right)_{1\leq i,j\leq r}$. Let  $(\delta_i,\delta'_j)$ be a pair of $X$-curvettes for $(\E_i,\E_j)$.
Then, $\delta_i\cdot\delta'_j=\frac{b_{i j}}{d(\delta_i) d(\delta'_j)}$.
\end{prop}

\begin{proof}
Let $\gamma'_i$ be a generic $\hat{X}$-curvette. Since $\gamma'_i$ and
$d(\gamma_i)\gamma_i$ are equivalent Weil divisors, we can assume
that $d(\gamma_i)=1$. We have $\pi^*(\delta_i)=\gamma_i+\sum_{j=1}^n c_{i j} \E_j$.
Note that $\gamma_i\cdot\E_j=\delta_{ij}$ ($\delta_{ij}$ being the Kronecker delta).

For a generic $\gamma'_j$ we have $\delta'_j\cdot\delta_i=\pi^*(\delta'_j)\cdot\pi^*(\delta_i)=
\gamma'_j\cdot\pi^*(\delta_i)=c_{i j}$.
Since
$$
\delta_{ik}=\gamma_i\cdot\E_k= \big{(}\pi^*(\delta_i)-\sum_{j=1}^n c_{i j} \E_j\big{)}\cdot \E_k=
- \sum_{j=1}^n (\delta_i\cdot\delta'_j) (\E_j\cdot \E_k),
$$ 
we deduce the result.
\end{proof}

\section{B{\'e}zout's Theorem for Weighted Projective Planes}

For a given weight vector $\w = (p,q,r) \in \N^3$ and an action on~$\C^3$ of
type $(d;a,b,c)$, consider the quotient weighted projective plane
$\bP^2_{\w}(d;a,b,c) := \bP^2_{\w} / \mu_d$ and the projection morphism
$\tau_{(d;a,b,c),\w}: \bP^2 \to \bP^2_{\w}(d;a,b,c)$ defined by
\begin{equation}\label{defi_tau}
\tau_{(d;a,b,c),\w} ( [x: y: z] ) = [x^p: y^q: z^r]_{\w,d}.
\end{equation}

The space $\bP^2_{\w}(d;a,b,c)$ is a variety with abelian quotient singularities;
its charts are obtained as in Section~\ref{weighted_projective_space}. The {\em
degree of a $\mathbb{Q}$-divisor on $\bP^2_{\w}(d;a,b,c)$} is the degree of its
pull-back under the map $\tau_{(d;a,b,c),\w}$, that is, by definition,
$$
  D \in \text{$\mathbb{Q}$-$\Div \left( \bP^2_{\w}(d;a,b,c) \right)$}, \quad
\deg_{\w} (D) := \deg \left( \tau_{(d;a,b,c),\w}^{*} (D) \right).
$$
Thus if $D = \{ F = 0 \}$ is a $\mathbb{Q}$-divisor on $\bP^2_{\w}(d;a,b,c)$
given by a $\w$-homogeneous polynomial that indeed defines a zero set on the
quotient projective space, then $\deg_{\w} (D)$ is the classical degree, denoted
by $\deg_{\w} (F)$, of the quasi-homogeneous polynomial.

%
%

The following result can be stated in a more general setting. However, it is
presented in this way to keep the exposition as simple as possible.

\begin{lemma}\label{computing_number_of_sheets}
The degree of the projection $\pr: \C^2 \longrightarrow X \big(
\begin{smallmatrix} d \, ; & a & b \\ e \, ; & r & s \end{smallmatrix} \big)$ is
given by the formula
$$
\displaystyle \frac{d \cdot e}{\gcd \big[ d \cdot \gcd(e,r,s),\ e \cdot
\gcd(d,a,b),\ as-br \big]}.
$$
\end{lemma}

\begin{proof}
Assume $\gcd(d,a,b) = \gcd(e,r,s) = 1$; the general formula is obtained easily
from this one.

The degree of the required projection $\C^2 \to X \big( \begin{smallmatrix} d \,
; & a & b \\ e \, ; & r & s \end{smallmatrix} \big)$ is $\frac{de}{\ell}$,
where~$\ell$ is the order of the abelian group
$$
H = \left\{ (\xi,\eta) \in \mu_d \times \mu_e \mid \xi^a \eta^r =1,\ \xi^b
\eta^s = 1 \right\} \lhd ( \mu_d \times \mu_e ).
$$

To calculate $\ell$, consider $(\xi, \eta) \in \mu_d \times \mu_e$ and solve the
system $\xi^a \eta^r = 1$, $\xi^b \eta^s = 1$.
Raising both equations to the $e$-th power, one obtains $\xi^{ae} = 1$ and
$\xi^{be} = 1$. Hence,
$$
\xi \in \mu_d \cap \mu_{ae} \cap \mu_{be} = \mu_{\gcd(d,ae,be)} =
\mu_{\gcd(d,e)}.
$$
Note that the assumption $\gcd(d,a,b)=1$ was used in the last equality.
Analogously, it follows that $\eta \in \mu_{\gcd(d,e)}$, provided that
$\gcd(e,r,s)=1$.

Thus, there exist $i,j \in \{ 0,1,\ldots,\gcd(d,e)-1 \}$ such that $\xi =
\zeta^{i}$ and $\eta = \zeta^{j}$, where $\zeta$ is a fixed $(d,e)$-th primitive
root of unity. Now the claim is reduced to finding the number of solutions of the
system of congruences
$$
\left\{ \begin{array}{rcl}
a i + r j & \equiv & 0 \\
b i + s j & \equiv & 0
\end{array}\right. \ \big( \text{mod } \gcd(d,e) \big).
$$
This is known to be $\gcd(d,e,as-br)$ and the proof is complete.
\end{proof}

\begin{prop}\label{bezout_th_P2w-mu_d}
Using the notation above, let us denote by $m_1$, $m_2$, $m_3$ the determinants
of the three minors of order $2$ of the matrix $\big( \begin{smallmatrix} p & q
& r \\ a & b & c \end{smallmatrix} \big)$. 
Denote $e: = \gcd ( d, m_1, m_2, m_3 )$.

Then the intersection number of two $\mathbb{Q}$-divisors on
$\bP^2_{\w}(d;a,b,c)$ is
$$
D_1 \cdot D_2 = \frac{e}{dpqr} \deg_{\w}(D_1) \deg_{\w}(D_2) \in \mathbb{Q}.
$$

In particular, the self-intersection number of a $\mathbb{Q}$-divisor is given
by $D^2 = \frac{e}{dpqr} \deg_{\w} (D)^2$. Moreover, if $|D_1| \nsubseteq
|D_2|$, then $|D_1| \cap |D_2|$ is a finite set of points and
\begin{equation}\label{eq_bezout_th_P2w-mu_d}
\frac{e}{dpqr} \deg_{\w}(D_1) \deg_{\w}(D_2) = \sum_{P \in |D_1| \cap |D_2|}
(D_1 \cdot D_2)_P.
\end{equation}
\end{prop}

\begin{proof}
For simplicity, let us just write $\tau$ for the map defined in \eqref{defi_tau} omitting the subindex. Note that $\tau$
is a proper morphism between two irreducible $V$-manifolds of dimension $2$.
Thus by Theorem~\ref{formula_pull-back}\eqref{formula_pull-back2} and the classical B{\'e}zout's theorem on
$\bP^2$, see \ref{classic_bezout}, one has the following sequence of equalities, 
$$
\deg(\tau)\! (D_1 \cdot D_2)\! =\! \tau^{*}(D_1) \cdot \tau^{*}(D_2)\!  =\! \deg \left(
\tau^{*} ( D_1) \right) \deg \left( \tau^{*} (D_2) \right) = \deg_{\w}(D_1)
\deg_{\w}(D_2). 
$$
The rest of the proof is the computation of $\deg(\tau)$; the final part is a
consequence of Remark \ref{bezoutglolo}.

In the first chart $\tau$ takes the form $\C^2 \to X \big( \begin{smallmatrix} p
\hspace{0.21cm} ; & q & r \\ pd \, ; & m_1 & m_2 \end{smallmatrix} \big)$,
$(y,z) \mapsto [(y^q, z^r)]$. By decomposing this morphism into $\C^2 \to \C^2$,
$(y,z) \mapsto (y^q,z^r)$ and the natural projection $\C^2 \to X \big(
\begin{smallmatrix} p \hspace{0.21cm} ; & q & r \\ pd \, ; & m_1 & m_2
\end{smallmatrix} \big)$, $(y,z) \mapsto [(y, z)]$, one obtains
$$
\deg(\tau) = qr \cdot \deg \left[ \C^2 \stackrel{\pr}{\longrightarrow} X \big(
\begin{smallmatrix} p \hspace{0.21cm} ; & q & r \\ pd \, ; & m_1 & m_2
\end{smallmatrix} \big) \right].
$$
The determinant of the corresponding matrix is $q m_2 - r m_1 = p m_3$. From
Lemma~\ref{computing_number_of_sheets} the latter degree is
$$
  \frac{p \cdot pd}{\gcd \big( p \cdot \gcd(pd, m_1, m_2), pd, p m_3 \big)} =
  \frac{dp}{\gcd \big( d, m_1, m_2, m_3 \big)}, 
$$
and hence the proof is complete.
\end{proof}

\begin{cor}
Let $X$, $Y$, $Z$ be the Weil divisors on $\bP^2_{\w}(d;a,b,c)$ given by $\{ x=0
\}$, $\{ y=0 \}$ and $\{ z=0 \}$, respectively. Using the notation of
Proposition{\rm~\ref{bezout_th_P2w-mu_d}} one has:
\begin{enumerate}[\rm(1)]
\item $\displaystyle X^2 = \frac{ep}{dqr}$, \quad $\displaystyle Y^2 =
\frac{eq}{dpr}$, \quad $\displaystyle Z = \frac{er}{dpq}$.

\item $\displaystyle X \cdot Y = \frac{e}{dr}$, \quad $\displaystyle X \cdot Z =
\frac{e}{dq}$, \quad $\displaystyle Y \cdot Z = \frac{e}{dp}$. $\hfill \Box$
\end{enumerate}
\end{cor}

\begin{remark}

%

 If $d=1$ then $e = 1$ too and the formulas above become a bit simpler. In
particular, one obtains the classical B{\'e}zout's theorem on weighted projective
planes, (the last equality holds if $|D_1| \nsubseteq |D_2|$ only) \vspace{0.1cm}
$$
\qquad \quad \ D_1 \cdot D_2 = \frac{1}{pqr} \deg_{\w}(D_1) \deg_{\w}(D_2) =
\sum_{P \in |D_1| \cap |D_2|} (D_1 \cdot D_2)_P.
$$

%
%

\end{remark}


\begin{ex}
Let us take again Example~\ref{ex-hirzebruch1}. The exceptional component $E$
has self-intersection $-\frac{q}{p r}$. Since the curve $y=0$ does not pass through
$P$ its self-intersection is the one in $\bP^2_\w$, i.e.~$\frac{q}{p r}$. 
The fibers~$F$ of $\pi$ have self-intersection~$0$;
a generic fiber is obtained as follows. Consider a curve~$L$ of equation $x^r-z^p=0$ in $\bP^2_\w$.
Then $\pi^*L=F+\frac{pr}{q} E$ and $F^2=0$. The surface $\hat{X}_P$ looks like
a Hirzebruch surface of index $\frac{q}{p r}$. 
\end{ex}

\section{Application to Jung resolution method}

One of the main reasons to work with $\Q$-resolutions of singularities is the fact that they are
much simpler from the combinatorial point of view and they essentially provide
the same information as classical resolutions. In the case of embedded resolutions, there
are two main applications. One of them is concerned with the study of the Mixed Hodge Structure and the topology of the Milnor fibration, see~\cite{Steenbrink77}.
The other one is the Jung method to find abstract resolutions, see \cite{Jung51} 
and a modern exposition~\cite{Laufer71} by Laufer.

The study of the Mixed Hodge Structure is related to a process called the semistable resolution
which introduces abelian quotient singularities and $\mathbb{Q}$-normal crossing divisors.
The work of the second author in his thesis guarantees that one can substitute embedded resolutions by
embedded $\Q$-resolutions obtaining the same results.
As for the Jung method, we will explain the usefulness of $\Q$-resolutions at the time they are presented.

\begin{nothing}\textbf{Classical Jung Method.}
Let
$H\in(\C^{n+1},0)$ be a hypersuface singularity defined by a Weiestra{\ss} polynomial
$f(x_0,x_1,\dots,x_n)\in\C\{x_1,\dots,x_n\}[x_0]$. Let $\Delta\in\C\{x_1,\dots,x_n\}$
be the discriminant of $f$. We consider the projection $\pi:H\to(\C^n,0)$ which is an $n$-fold
covering ramified along $\Delta$. Let
$\sigma:X\to(\C^n,0)$ be an embedded resolution of the singularities of $\Delta$.
Let $\hat{X}$ be the pull-back of $\sigma$ and $\pi$. In general, this space has non-normal singularities. Denote by $\nu:\tilde{X}\to\hat{X}$ its normalization.
$$
\xymatrix{
\bar{X} \ar[r]^{\tau} & \tilde{X} \ar[r]^{\nu} \ar[d]_{\tilde{\pi}} 
\ar@/^0.5cm/[rr]^{\tilde{\sigma}} \ar@{}[rd]|{\#} & \hat{X} \ar[r] 
\ar[d] \ar@{}[rd]|{\#} & H \ar[d]^{\pi}\\
& X \ar@{=}[r] & X \ar[r]_{\sigma} & \C^n
}
$$
There are two mappings issued from $\tilde{X}$: $\tilde{\pi}:\tilde{X}\to X$ and
 $\tilde{\sigma}:\tilde{X}\to H$. The map
$\tilde{\pi}$ is an $n$-fold covering whose branch locus is contained in $\sigma^{-1}(\Delta)$.
In general, $\tilde{X}$ is not smooth, it has abelian quotient 
singular points over the (normal-crossing) singular points
of $\sigma^{-1}(\Delta)$. Consider $\tau:\bar{X}\to\tilde{X}$ the resolution of $\tilde{X}$,
see~\cite{Fujiki74}. Then $\tilde{\sigma}\circ\tau:\bar{X}\to H$
is a good abstract resolution of the singularities of $H$.
\end{nothing}

\begin{nothing}\textbf{Jung $\Q$-method.}
In the previous method, $\tilde{\sigma}$ is a $\Q$-resolution of $H$. This is why
replacing $\sigma$ by an embedded $\Q$-resolution is a good idea. First, 
the process to obtain an embedded $\Q$-resolution is much shorter; we can reproduce
the process above and the space $\tilde{X}$ obtained only has abelian quotient singularities
and the exceptional divisor has $\mathbb{Q}$-normal crossings, 
i.e.~$\tilde{\sigma}$ is an abstract $\Q$-resolution of $H$, usually simpler than
the one obtained by the classical method.

If anyway, one is really interested in a standard resolution of $H$, the most direct way
to find the intersection properties of the exceptional divisor of $\tilde{\sigma}\circ\tau$
is to study the $\mathbb{Q}$-intersection properties of the exceptional
divisor of $\tilde{\sigma}$ and construct $\tau$ as a composition of weighted-blow ups.
\end{nothing}

We explain this process more explicitly in the case $n=2$. 
After the pull-back and the normalization process, the preimage
of each irreducible divisor $\E$ of $\Delta$ is a (possibly non-connected)
ramified covering of $\E$. In order to avoid technicalities to
describe these coverings, we restrict our attention to the cyclic case,
i.e.~$H=\{z^n=f(x,y)\}$. 

In this case the reduced structure of $\Delta$ is the one of
$f(x,y)=0$. We consider the minimal $\Q$-resolution of $\Delta$, which is
obtained as a composition of weighted blow-ups following the Newton process,
see~\ref{puiseux}.

Let $\E$ be an irreducible component of $\sigma^{-1}(\Delta)$ with multiplicity $s:=m_\E$. 

\begin{nothing}\textbf{Generic points of $\E$.}
Consider
a generic point $p\in\E$ with local coordinates $(u,v)$ such that $v=0$ is $\E$ and
$(f\circ\sigma)(u,v)=v^{s}$. Note that $p$ has only one preimage in $\hat{X}$; $\bar{X}$ looks in
a neighborhood of this preimage like $\{(u,v,z)\in\C^3\mid z^n=v^s\}$. The normalization of this space
produces $\gcd(s,n)$ points which are smooth. Then, the preimage of $\E$ in $\tilde{X}$
is (possibly non-connected) $\gcd(s,n)$-sheeted cyclic covering ramified on the singular points
of $\E$ in $\sigma^{-1}(\Delta)$; the number of connected components
and their genus will be described later. Note also that $\tilde{X}$ is smooth over the smooth
part of $\E$ in $\sigma^{-1}(\Delta)$. 
\end{nothing}

\begin{remark}
In the general (non-cyclic) case, the local equations can be more complicated but
we always have that the preimage of $\E$ in $\tilde{X}$
is a possibly non-connected covering ramified on the double points
of $\E$ in $\sigma^{-1}(\Delta)$ and $\tilde{X}$ is smooth over the non-ramified
part of $\E$.
\end{remark}

Let $p\in\Sing^0(\E)$ of normalized type $(d;a,b)$. Since $d$ divides $s$, let us denote:
$$
s_0:=\frac{s}{d},
g:=\gcd(n,s_0), n_1:=\frac{n}{g}, s_1:=\frac{s_0}{g},
e:=\gcd(n_1,d), n_2:=\frac{n_1}{e},
d_1:=\frac{d}{e}.
$$

\begin{lemma}\label{lema-sing0}
The preimage of $p$ under $\tilde{\sigma}$ consists of $g$ points
of type $(d_1;a n_2,b)$. 
\end{lemma}

\begin{proof}
The local model
of $\hat{X}$ around the preimage of $p$ is of the type
$\{([(u,v)],z)\in X(d;a,b)\times\C\mid z^n=v^s\}$. Consider
$$
z^n-v^s=\prod_{\zeta_g\in\mu_g} (z^{n_1}-\zeta_g v^{s_1 d}).
$$
Note that each factor is well defined in $X(d;a,b)\times\C$, and hence
the normalization is composed by $g$ copies of the normalization
of $z^{n_1}= v^{s_1 d}$. 

In $\C^3$ the space $z^{n_1}= v^{s_1 d}$ has $e$ irreducible components
and the action of $\mu_d$ permutes cyclically these components. Hence
the quotient of this space by $\mu_d$ is the same as the
quotient of $z^{n_2}= v^{s_1 d_1}$ by the action of $\mu_{d_1}$
defined by $\zeta_{d_1}\cdot(u,v,z)\mapsto(\zeta_{d_1}^a u,\zeta_{d_1}^b v,z)$.
The normalization of $z^{n_2}= v^{s_1 d_1}$ is given by
$$
(u,t)\mapsto (u,t^{n_2},t^{s_1 d_1})
$$
and the induced action of $\mu_{d_1}$ is defined by 
$$
\zeta_{d_1}\cdot(u,t)\mapsto (\zeta_{d_1}^a u,\zeta_{d_1}^{b\alpha} t),\quad
\alpha n_2\equiv 1\mod d_1.
$$
The result follows since $(d_1;a,b\alpha)=(d_1;a n_2,b)$. 
\end{proof}

Let us  consider now a double point $p$ of type
$X(d;a,b)$ (normalized), $(\E_1,\E_2)$ and let $r,s$
be the multiplicities of $\E_1,\E_2$. 
Some notation is needed:
$$
m_0:=\frac{a r+b s}{d}, n_1:=\frac{n}{g},
r_1:=\frac{r}{g}, s_1:=\frac{s}{g}, m_1:=\frac{m_0}{g}.
$$
Note that $a r_1+b s_1=m_1 d$. We complete the notation:
$$
e:=\gcd(n_1,r_1,s_1), n_2:=\frac{n_1}{e}, r_2:=\frac{r_1}{e},
s_2:=\frac{s_1}{e}, d_1:=\frac{d}{e}.
$$ 
Since $\gcd(m_1,e)=1$, $e$ divides $d$ and then $d_1\in\Z$.
Note that $a r_2+b s_2=m_1 d_1$.
Since $\gcd(n_2,r_2,s_2)=1$, one fixes $k,l\in\Z$ such that $m_1+ k r_2+ l s_2\equiv 0\mod n_2$
and denote:
$$
a':=a+k d_1,\quad b':=b+l d_1.
$$
\begin{lemma}\label{lema-doublepoint}
The preimage of $p$ under $\tilde{\sigma}$ consists of $g$ points
of type 
$$
X
\left(
\begin{matrix}
d_1 n_2\\
n_2 
\end{matrix}
\right|
\left.
\begin{matrix}
a'&b'\\
s_2&-r_2 
\end{matrix}
\right).
$$
The type is not normalized.
\end{lemma}

\begin{proof}
The local model
of $\hat{X}$ over $p$ is
$\{([(u,v)],z)\in X(d;a,b)\times\C\mid z^n=u^r v^s\}$.
We have
$$
z^n-u^r v^s=\prod_{\zeta_g\in\mu_g} (z^{n_1}-\zeta_g u^{r_1} v^{s_1}).
$$
Since each factor is well-defined in $X(d;a,b)\times\C$, 
the normalization is composed by $g$ copies of the normalization
of $z^{n_1}= u^{r_1} v^{s_1}$.

In $\C^3$ the space $z^{n_1}= u^{r_1} v^{s_1}$ has $e$~irreducible components
and the action of $\mu_d$ permutes cyclically these components. Hence
the quotient of this space by $\mu_d$ is the same as the
quotient of $z^{n_2}= u^{r_2} v^{s_2}$ by the action of $\mu_{d_1}$
defined by $\zeta_{d_1}\cdot(u,v,z)\mapsto(\zeta_{d_1}^a u,\zeta_{d_1}^b v,z)$.

Note that $a,b$ can be replaced by $a',b'$ in the
action of $\mu_{d_1}$. Moreover, $D:=a' r_2+b' s_2\equiv 0\mod n_2$. The map
$$
(t,w)\mapsto(t^{n_2},w^{n_2},t^{r_2} w^{s_2}) 
$$ 
parametrizes (not in a biunivocal way) the space $z^{n_2}= u^{r_2} v^{s_2}$.
The action of $\mu_{n_2 d_1}$ defined by
$$
\zeta_{n_2 d_1}\cdot(t,w)\mapsto(\zeta_{n_2 d_1}^{a'} t,\zeta_{n_2 d_1}^{b'} w)
$$
lifts the former action of $\zeta_{d_1}$. The normalization 
of the quotient of $z^{n_2}= u^{r_2} v^{s_2}$ by the action of $\mu_{d_1}$ is deduced to be of (non-normalized) type
\begin{equation*}
X
\left(
\begin{matrix}
d_1 n_2\\
n_2 
\end{matrix}
\right|
\left.
\begin{matrix}
a'&b'\\
s_2&-r_2 
\end{matrix}
\right).\qedhere
\end{equation*}
\end{proof}

\begin{remark}
It is easier to normalize this type case by case,
but at least a method to present it as a cyclic type is shown here. Let $\mu:=\gcd(a',d_1 s_2)$ and let
$\beta,\gamma\in\Z$ such that $\mu=\beta a'+\gamma d_1 s_2$.
Note that $\mu$ divides $D$.
Then the preceding type is isomorphic (via the identity) to
$$
X
\left(
\begin{matrix}
n_2\\
n_2\\
d_1 n_2 
\end{matrix}
\right|
\left.
\begin{matrix}
0&\gamma\frac{D}{\mu}\\
0&-\beta\frac{D}{\mu}\\ 
\mu&\beta b'-\gamma d_1 r_2
\end{matrix}
\right)=
X
\left(
\begin{matrix}
n_2\\
d_1 n_2 
\end{matrix}
\right|
\left.
\begin{matrix}
0&\frac{D}{\mu}\\ 
\mu&\beta b'-\gamma d_1 r_2
\end{matrix}
\right)
$$
since $\gcd(\beta,\gamma)=1$.
Let $h:=\gcd(n_2,\frac{D}{\mu})$.
Then, this space is isomorphic to
$X(d_1 n_2; \alpha,(\beta b'-\gamma d_1 r_2)\frac{n_2}{h})$.
If $j:=\gcd(\mu,\frac{n_2}{h})$, then it is isomorphic to
the space $X(d_1 h; \frac{\mu}{j},\beta b'-\gamma d_1 r_2)$
(maybe non normalized).
 \end{remark}

The following statement summarizes the results for each irreducible component
of the divisor.

\begin{lemma}\label{lema-component}
Let $\E$ be an exceptional component of $\sigma$ with multiplicity $s$,
$m:=\gcd(s,n)$.
Let $\Sing(\E)$ be the union of $\Sing^0(\E)$ with the double points
of $\sigma^{-1}(\Delta)$ in $\E$. Let $\nu$ be the $\gcd$ of $s$
and the values $g_P$ for each $P\in\Sing(\E)$ obtained in 
Lemmas{\rm~\ref{lema-sing0}} and{\rm~\ref{lema-doublepoint}}.
Then, $\tilde{\sigma}^{-1}(\E)$ consists of $\nu$ connected
components. Each component is an $(\frac{s}{\nu})$-fold cyclic
covering whose genus is computed using Riemann-Hurwitz formula
and the self-intersection of each component is
$\frac{m^2 \eta}{n \nu}$ if $\eta=\E^2$.  
\end{lemma}

\begin{proof}
Only the self-intersection statement needs a proof. Let
$\tilde{\E}:=\tilde{\sigma}^{-1}(\E)$. Then
$\tilde{\sigma}^{*}(\E)=\frac{n}{m}\tilde{\E}$. Hence:
$$
\tilde{\E}^2=\frac{m^2}{n^2}\tilde{\sigma}^{*}(\E)^2=
\frac{m^2}{n}(\E)^2=\frac{m^2\eta}{n}.
$$
Since $\tilde{\E}$ has $\nu$ disjoint components
related by an automorphism of $\tilde{X}$, the result follows.
\end{proof}

\begin{ex}
Let us consider the singularity $z^n-(x^2+y^3)(x^3+y^2)=0$, $n>1$. As it was shown in
Example~\ref{example_2}, the minimal $\Q$-embedded resolution of $(x^2+y^3)(x^3+y^2)=0$
has two exceptional components $\E_1,\E_2$. Each component has multiplicity $10$,
self-intersection $-\frac{3}{10}$,
intersects the strict transform at a smooth point and has one singular point
of type $(2;1,1)$. The two components intersect at a double point of type 
$(5;2,-3)=(5;1,1)$. Let us denote $g_p(n)$ the previous numbers for a given~$n$.
The computations are of four types depending on $\gcd(n,10)=1,2,5,10$.

Let us fix one of the exceptional components, say $\E_1$, since
they are symmetric. Before studying separately each case, let $p_0$ 
be the intersection point of $\E_1$ with the strict transform, then its preimage
is the normalization of $z^n-x y^{10}=0$ which is of
type $(n;-10,1)$. In particular $g_p(n)=1$ and
$\nu_{\E_1}=1$, i.e.~$\tilde{\E_1}:=\tilde{\sigma}^{-1}(\E_1)$
is irreducible. Let us denote $p_1:=\E_1\cap\E_2$.

\begin{caso}
$\gcd(n,10)=1$.
\end{caso}

Let us study first
the preimage over a generic point of $\E_1$, which will be the normalization
of $z^n-y^{10}=0$, i.e.~one point. By Lemma~\ref{lema-component},
$\tilde{\E_1}^2=-\frac{3}{10 n}$ and $\tilde{\E}_1$ is rational. The preimage
of $p_0$ is of reduced type $(n;-10,1)$.

Let $p\in\Sing^0(\E_1)$. It is
of type $(2;1,1)$. Applying Lemma~\ref{lema-sing0},
one obtains that it
is of type $(2;11,1)=(2;1,1)$.

One has $g_{p_1}(n)=e=1$.
Following the notation previous to Lemma~\ref{lema-doublepoint}, we choose $k=l\in\Z$ such that $5 k+1\equiv 0\mod n$.
A type
$$
\left(
\begin{matrix}
5 n\\
n 
\end{matrix}
\right|
\left.
\begin{matrix}
1+5 k&1+5 k\\
10&-10
\end{matrix}
\right)=
\left(
\begin{matrix}
5 n\\
n 
\end{matrix}
\right|
\left.
\begin{matrix}
1+5 k&1+5 k\\
1&-1
\end{matrix}
\right)=
\left(
\begin{matrix}
5 n\\
5 n 
\end{matrix}
\right|
\left.
\begin{matrix}
1+5 k&1+5 k\\
5&-5
\end{matrix}
\right),
$$
is obtained, which is of type $(5 n; 1,10 k+1)$;
since $(10 k+1)^2\equiv 1\mod 5 n$, this type is symmetric and normalized.
Then, the minimal embedded $\Q$-resolution of the surface singularity
consists of two rational divisors of self-intersection $-\frac{3}{10 n}$,
with a unique double point of type $X(5 n;1;1+10 k)$ and each divisor
has two other singular points, one double and the other one
of type $X(n;-10,1)$.

\begin{figure}[h t]
\centering\includegraphics{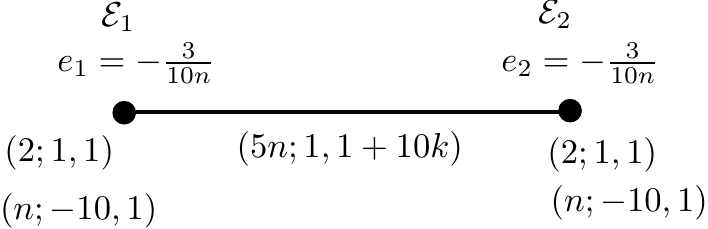}
\caption{Dual graph for $z^n= (x^2+y^3)(x^3+y^2)$, $\gcd(n,10)=1$, $5 k+1\equiv 0\mod n$.}
\label{fig-grafo-jung1}
\end{figure}

\begin{caso}
$\gcd(n,10)=2$.
\end{caso}

The preimage over a generic point of $\E_1$, which will be the normalization
of $z^n-y^{10}=0$, i.e. $\tilde{\E_1}:=\tilde{\sigma}^{-1}(\E_1)$ is
a $2$-fold covering of $\E_1$. The point $p_0$ is a ramification point of the covering (with one preimage) and
it is of type $(n;-10,1)=(\frac{n}{2};-5,1)$.

Let $p\in\Sing^0(\E_1)$. Since $s_0=2$, $g_p(n)=1$ and $e=2$, applying Lemma~\ref{lema-sing0}, one has $d_1=1$.
There is only one preimage and it is a smooth point.

Let us finish with $p_1$. In this case, $g_{p_1}(n)=2$, $n_1:=\frac{n}{2}$, and
$e=1$. It can be chosen $k=l\in\Z$ such that $5 k+1\equiv 0\mod n_1$.
Using the same computations as in the previous case,
two points of type $X(5 n_1; 1,10 k+1)$ are obtained.

Using Riemann-Hurwitz formula $\tilde{\E_1}$
is irreducible and rational; since $\tilde{\sigma}^*(\E_1)=5 \tilde{\E_1}$
one has that $\tilde{\E_1}^2=-\frac{3}{5 n_1}$.
Then, the minimal embedded $\Q$-resolution of the surface singularity
consists of two rational divisors,
with two double points of type $X(5 n_1;1,1+10 k)$ and each divisor
has another singular point
of type $X(n_1;-5,1)$. Note that the graph is not a tree.

\begin{figure}[h t]
\centering\includegraphics{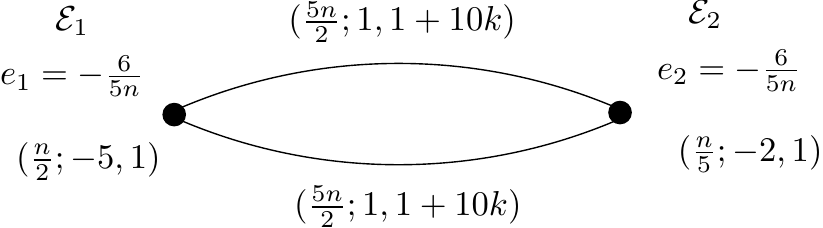}
\caption{Dual graph for $z^n= (x^2+y^3)(x^3+y^2)$, $\gcd(n,10)=2$, $5 k+1\equiv 0\mod \frac{n}{2}$.}
\label{fig-grafo-jung2}
\end{figure}

\begin{caso}
$\gcd(n,10)=5$.
\end{caso}

The preimage over a generic point of $\E_1$, which will be the normalization
of $z^n-y^{10}=0$, i.e. $\tilde{\E_1}:=\tilde{\sigma}^{-1}(\E_1)$ is
a $5$-fold covering of $\E_1$.
As above, $p_0$ is a ramification point of the covering (with one preimage) and
it is of type $(n;-10,1)=(\frac{n}{5}:-2,1)$.

Let $p\in\Sing^0(\E_1)$. One has $g_p(n)=5$ and $d_1=2$.
Hence the covering does not ramify at $p$ and
its preimage consists of $5$ points of type~$(2;1,1)$.

In the case of $p_1$ we have $g_{p_1}(n)=1$, $e=5$, $n_2=\frac{n}{5}$
and $d_1=1$. Hence a point of type $X(n_2; 1,-1)$ is obtained.

As a consequence, $\tilde{\E_1}$ is
rational and $\tilde{\E_1}^2=-\frac{3}{2 n_2}$.
Then, the minimal embedded $\Q$-resolution of the surface singularity
consists of two rational divisors,
with one double point of type $X(n_2;1;-1)$ and each divisor
has another singular point 
of type $X(n_2;-2,1)$ and five double points.

\begin{figure}[h t]
\centering\includegraphics{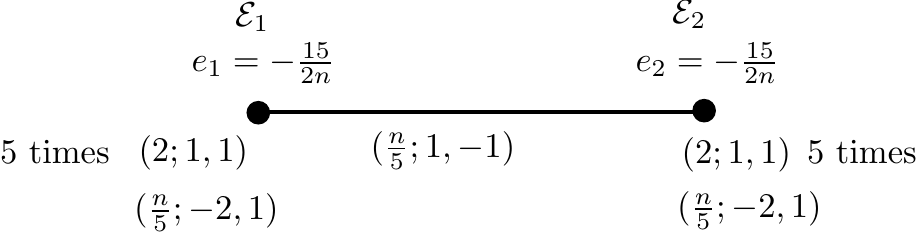}
\caption{Dual graph for $z^n= (x^2+y^3)(x^3+y^2)$, $\gcd(n,10)=5$.}
\label{fig-grafo-jung3}
\end{figure}

\begin{caso}
$\gcd(n,10)=10$.
\end{caso}

The preimage over a generic point of $\E_1$, which will be the normalization
of $z^n-y^{10}=0$, i.e. $\tilde{\E_1}:=\tilde{\sigma}^{-1}(\E_1)$ is
a $10$-fold covering of $\E_1$.
The point $p_0$ is a ramification point of the covering (with one preimage) and
it is of type $(n;-10,1)=(\frac{n}{10}:-1,1)$.

Let $p\in\Sing^0(\E_1)$. One has $g_p(n)=5$ and $d_1=1$.
Hence the preimage of $p$ consists of $5$ smooth points.

Finally one has $g_{p_1}(n)=2$, $e=5$, $n_2=\frac{n}{10}$
and $d_1=1$. Hence a point of type $X(n_2; 1,-1)$ is obtained.

Using Riemann-Hurwitz, $\tilde{\E_1}$ has genus~$2$; since $\tilde{\sigma}^*(\E_1)=\tilde{\E_1}$, then $\tilde{\E_1}^2=-\frac{3}{n_2}$.
Then, the minimal embedded $\Q$-resolution of the surface singularity
consists of two divisors of genus~$2$,
with one double point of type $X(n_2;1;-1)$ and each divisor
has another singular point
of type $X(n_2;-1,1)$.
\begin{figure}[h t]
\centering\includegraphics{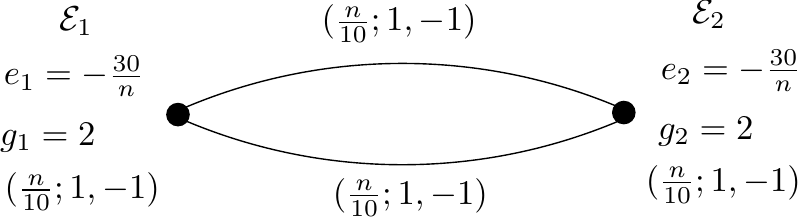}
\caption{Dual graph for $z^n= (x^2+y^3)(x^3+y^2)$, $\gcd(n,10)=10$.}
\label{fig-grafo-jung4}
\end{figure}
\end{ex}

As a final application, intersection theory and weighted
blow-ups are essential tools to construct a resolution from a $\Q$-resolution.
Note that even when one uses the classical Jung method, this step is needed. 
The resolution of cyclic quotient singularities
for surfaces is known, see the works by Jung~\cite{Jung08},
Hirzebruch~\cite{Hirzebruch53}, and the exposition in the book
by Hirzebruch-Neumann-Koh~\cite{hnk}.

This resolution process uses the theory of continuous fractions.
We illustrate the use of weighted blow-ups to solve these singularities 
in two ways.

First, let $X:=X(d;a,b)$, where $d,a,b$ are pairwise coprime, $d>1$,
and $1\leq a,b<d$. Then the $(a,b)$-blow-up of $X$ produces a new space
with an exceptional divisor (of self-intersection $-\frac{d}{a b}$)
and two singular points of type $(a;-d,b)$ and $(b;a-d)$. Since the index
of these singularities is less than~$d$ we finish by induction. Note
that if we have a compact divisor passing through the singular point,
it is possible to compute the self-intersection multiplicity of the strict transform, see Proposition~\ref{formula_self-intersection}.

The second way allows us to recover the Jung-Hirzebruch resolution. Recall briefly
the notion of continuous fraction. Let $s\in\mathbb{Q}$, $r>1$. The continuous
fraction associated with $s$ is a tuple of integers $cf(s):=[q_1,\dots,q_n]$, $q_j>1$,
defined inductively as follows:
\begin{itemize}
\item If $s\in\Z$ then $cf(s):=[s]$.
\item If $s\notin\Z$, write $s=\frac{d}{k}$ in reduced form. Consider the \emph{excess division algorithm}
$d=q k-r$, $q,r\in\Z$, $0<r<k$. Then, $cf(s):=[q,cf(\frac{k}{r})]$. 
\end{itemize}
 Hence, for instance, $[q_1,q_2,q_3]=q_1-\frac{1}{q_2-\frac{1}{q_3}}$.
\begin{prop}
Let $X:=X(d;1,k)$ be a normalized type with $1\leq k<d$ and let $cf(\frac{d}{k}):=[q_1,\dots,q_n]$. Then the 
exceptional locus of the resolution of $X$ consists of a linear chain of rational curves
with self-intersections $-q_1,\dots,-q_n$.
\end{prop}

\begin{proof}
As stated above, we perform the $(1,k)$-blow-up of $X$. We obtain 
an exceptional divisor $\E_1$ such that $\E_1^2=-\frac{d}{k}$.
If $k=1$, we are done. If $k>1$ then $\E_1$
contains a singular point $Y:=X(k;1,-d)$. We know
that $d=q_1 k-r$, $1<r<k$, and $cf(\frac{k}{r})=[q_2,\dots,q_n]$.
Since $r=q_1 k-d$, then $Y=X(k;1,r)$. We may apply induction hypothesis (in the length of $cf$) 
and the result follows if we obtain
the right self-intersection multiplicity of the first divisor.

The next blow-up is with respect to $(1,r)$. Following Proposition \ref{formula_self-intersection}, the self-intersection
of the strict transform of $\E_1$ equals $-\frac{d}{k}-\frac{r^2}{k\cdot 1\cdot r}=-\frac{d}{k}-\frac{r}{k}=-q_1$, since the divisor $\E_1$ is given by $\{y=0\}\subset X(k;1,r)$.
\end{proof}

\begin{remark}
The last part of the proof allows us to give the right way to pass
from a $\Q$-resolution to a resolution. The Jung-Hirzebruch method
gives the resolution of the cyclic singularities. Let $\E$ be an irreducible
component of the $\Q$-resolution with self-intersection $-s\in\mathbb{Q}$,
and let $\Sing(\E):=\{P_1,\dots,P_r\}$, with $P_i$ of type $(d_i;1,k_i)$,
$1\leq k_i<d_i$ and $\gcd(d_i,k_i)=1$. Then, the self-intersection number
of $\E$, assuming its local equation is $y=0$, can be computed as above. That is, one performs the weighted blow-ups of
type $(1,k_i)$ at each point, obtaining
$
-s-\sum_{i=1}^r \frac{k_i}{d_i}
$,
which must be an integer.
\end{remark}

In future work, we will use these methods to study curves
in weighted projective planes.


\begin{thebibliography}{10}

\bibitem{kjj-qcw}
E.~Artal, J.~Mart{\'i}n-Morales, and J.~Ortigas-Galindo, \emph{{C}artier and
  {W}eil divisors on varieties with quotient singularities}, Preprint available
  at \texttt{arXiv:1104.5628 [math.AG]}, 2011.

\bibitem{Baily56}
W.L. Baily, Jr., \emph{The decomposition theorem for {$V$}-manifolds}, Amer. J.
  Math. \textbf{78} (1956), 862--888.

\bibitem{Dolgachev82}
I.~Dolgachev, \emph{Weighted projective varieties}, Group actions and vector
  fields ({V}ancouver, {B}.{C}., 1981), Lecture Notes in Math., vol. 956,
  Springer, Berlin, 1982, pp.~34--71.

\bibitem{Fujiki74}
A.~Fujiki, \emph{On resolutions of cyclic quotient singularities}, Publ. Res.
  Inst. Math. Sci. \textbf{10} (1974/75), no.~1, 293--328.

\bibitem{Fulton98}
W.~Fulton, \emph{Intersection theory}, second ed., Ergebnisse der Mathematik
  und ihrer Grenzgebiete. 3. Folge. A Series of Modern Surveys in Mathematics,
  vol.~2, Springer-Verlag, Berlin, 1998.

\bibitem{Hirzebruch53}
F.~Hirzebruch, \emph{\"{U}ber vierdimensionale {R}iemannsche {F}l\"achen
  mehrdeutiger analytischer {F}unktionen von zwei komplexen
  {V}er\"anderlichen}, Math. Ann. \textbf{126} (1953), 1--22.

\bibitem{hnk}
F.~Hirzebruch, W.~D. Neumann, and S.~S. Koh, \emph{Differentiable manifolds and
  quadratic forms}, Marcel Dekker Inc., New York, 1971, Appendix II by W.
  Scharlau, Lecture Notes in Pure and Applied Mathematics, Vol. 4.

\bibitem{Jung08}
H.W.E. Jung, \emph{Darstellung der funktionen eines algebraischen k{\"o}rpers
  zweier unabh{\"a}ngigen ver{\"a}nderlichen $x$, $y$ in der umgebung einen
  stelle $x = a$, $y = b$}, vol. 133, 1908.

\bibitem{Jung51}
\bysame, \emph{Einf\"uhrung in die {T}heorie der algebraischen {F}unktionen
  zweier {V}er\"anderlicher}, Akademie Verlag, Berlin, 1951.

\bibitem{Laufer71}
H.B. Laufer, \emph{Normal two-dimensional singularities}, Princeton University
  Press, Princeton, N.J., 1971, Annals of Mathematics Studies, No. 71.

\bibitem{Mumford61}
D.~Mumford, \emph{The topology of normal singularities of an algebraic surface
  and a criterion for simplicity}, Inst. Hautes \'Etudes Sci. Publ. Math.
  (1961), no.~9, 5--22.

\bibitem{Prill67}
D.~Prill, \emph{Local classification of quotients of complex manifolds by
  discontinuous groups}, Duke Math. J. \textbf{34} (1967), 375--386.

\bibitem{Satake56}
I.~Satake, \emph{On a generalization of the notion of manifold}, Proc. Nat.
  Acad. Sci. U.S.A. \textbf{42} (1956), 359--363.

\bibitem{Steenbrink77}
J.H.M. Steenbrink, \emph{Mixed {H}odge structure on the vanishing cohomology},
  Real and complex singularities ({P}roc. {N}inth {N}ordic {S}ummer
  {S}chool/{NAVF} {S}ympos. {M}ath., {O}slo, 1976), Sijthoff and Noordhoff,
  Alphen aan den Rijn, 1977, pp.~525--563.

\bibitem{Veys97}
Willem Veys, \emph{Zeta functions for curves and log canonical models}, Proc.
  London Math. Soc. (3) \textbf{74} (1997), no.~2, 360--378.

\end{thebibliography}
%
\def\cprime{$'$}
\providecommand{\bysame}{\leavevmode\hbox to3em{\hrulefill}\thinspace}
\providecommand{\MR}{\relax\ifhmode\unskip\space\fi MR }
\providecommand{\MRhref}[2]{%
  \href{http://www.ams.org/mathscinet-getitem?mr=#1}{#2}
}
\providecommand{\href}[2]{#2}

\end{document}